\documentclass[letterpaper,12pt,reqno]{amsart}

\usepackage{amsmath, amsthm}
\usepackage{eucal}

\usepackage{amssymb}
\usepackage{amscd}
\usepackage{latexsym}
\usepackage{graphicx}
\usepackage{amsfonts}

\usepackage{url}

\usepackage{float}		

\usepackage[colorlinks=true, citecolor=blue, linktocpage]{hyperref}

\usepackage[margin=1in]{geometry}

\usepackage{subfigure}	

\newtheorem{theorem}{Theorem}[section] 
\newtheorem{lemma}[theorem]{Lemma}
\newtheorem{corollary}[theorem]{Corollary}

\newtheorem{proposition}[theorem]{Proposition}
\newtheorem{question}[theorem]{Question}

\theoremstyle{definition}
\newtheorem{example}[theorem]{Example}

\theoremstyle{definition}
\newtheorem{definition}[theorem]{Definition}

\theoremstyle{plain}

\theoremstyle{definition}
\newtheorem{remark}[theorem]{Remark}

\newcommand{\LT}[1]{\left\langle LT(#1)\right\rangle}
\newcommand{\X}{\textbf{x}^\alpha}
\newcommand{\Y}{\textbf{x}^\beta}

\newcommand{\IAD}{I_h^{AD}}

\DeclareMathOperator{\height}{ht}			

\input xy	
\xyoption{all}

\makeatletter
\newcommand\ackname{Acknowledgements}
\if@titlepage
  \newenvironment{acknowledgements}{%
      \titlepage
      \null\vfil
      \@beginparpenalty\@lowpenalty
      \begin{center}%
        \bfseries \ackname
        \@endparpenalty\@M
      \end{center}}%
     {\par\vfil\null\endtitlepage}
\else
  \newenvironment{acknowledgements}{%
      \if@twocolumn
        \section*{\abstractname}%
      \else
        \small
        \begin{center}%
          {\bfseries \ackname\vspace{-.5em}\vspace{\z@}}%
        \end{center}%
        \quotation
      \fi}
      {\if@twocolumn\else\endquotation\fi}
\fi
\makeatother

\begin{document}

\title{Generalizing Tanisaki's ideal via ideals of truncated symmetric functions}

\author{Aba Mbirika}
\address{Department of Mathematics, Bowdoin College, 8600 College Station, Brunswick, Maine, 04011-8486, U.S.A.}
\email{ambirika@bowdoin.edu}
\urladdr{\url{http://www.bowdoin.edu/~ambirika}}
\thanks{JT is partially supported by NSF grant DMS-0801554 and a Sloan Research Fellowship.}

 \author{Julianna Tymoczko}
 \address{Department of Mathematics and Statistics, Smith College, Clark Science Center, Northampton, Massachusetts, U.S.A.}
 \email{jtymoczko@smith.edu}
\urladdr{\url{http://www.math.smith.edu/~jtymoczko/}}



\date{\today}

\begin{abstract}
We define a family of ideals $I_h$ in the polynomial ring $\mathbb{Z}[x_1,\ldots,x_n]$ that are parametrized by Hessenberg functions $h$ (equivalently Dyck paths or ample partitions).  The ideals $I_h$ generalize algebraically a family of ideals called the Tanisaki ideal, which is used in a geometric construction of permutation representations called Springer theory.  To define $I_h$, we use polynomials in a proper subset of the variables $\{x_1,\ldots,x_n\}$ that are symmetric under the corresponding permutation subgroup.  We call these polynomials {\em truncated symmetric functions} and show combinatorial identities relating different kinds of truncated symmetric polynomials.  We then prove several key properties of $I_h$, including that if $h>h'$ in the natural partial order on Dyck paths then $I_{h} \subset I_{h'}$, and explicitly construct a Gr\"{o}bner basis for $I_h$.  We use a second family of ideals $J_h$ for which some of the claims are easier to see, and prove that $I_h = J_h$.  The ideals $J_h$ arise in work of Ding, Develin-Martin-Reiner, and Gasharov-Reiner on a family of Schubert varieties called partition varieties.  Using earlier work of the first author, the current manuscript proves that the ideals $I_h = J_h$ generalize the Tanisaki ideals both algebraically and geometrically, from Springer varieties to a family of nilpotent Hessenberg varieties.
\end{abstract}

\maketitle

\setcounter{tocdepth}{1}
\tableofcontents


\section{Introduction}
\label{intro}

Symmetric functions are polynomials in $\mathbb{Z}[x_1, x_2, \ldots, x_n]$ that are fixed by the natural action of the permutation group $S_n$ on the variables $\{x_1, x_2, \ldots, x_n\}$.  They are fundamental objects in many fields of mathematics; for instance, surveys by Macdonald, Stanley, and Fulton describe symmetric functions from algebraic~\cite{Mac95}, combinatorial~\cite{RS2}, and geometric perspectives~\cite{Fult97}, including applications to Schubert calculus and geometric representation theory.

We study functions that are symmetric in a subset of the variables $\{x_1, \ldots, x_n\}$, which we call {\em truncated symmetric functions}.\footnote{Biagioli-Faridi-Rosas call these functions {\em partially symmetric functions} \cite{Biag07}, \cite{Biag08}; we avoid this terminology because it refers to something slightly different in computer science.}  The idea of truncated symmetric functions arises naturally in Schubert calculus, for instance in the {\em stability} property of Schubert polynomials and in the quantum cohomology of flag varieties (e.g.,~\cite{FomKir96}, \cite{FGP97}, \cite{Pos99}); they also appear in Springer theory, described below (e.g.,~\cite{dCP}, \cite{Tani}, \cite{Biag08}).  Our main concerns are the {\em truncated elementary symmetric functions} $e_d(x_1, \ldots, x_r)$, defined  as the sum of all squarefree monomials of degree $d$ in the variables $x_1,\ldots,x_r$, and the {\em truncated complete symmetric functions} $\tilde{e}_d(x_{r'}, x_{r'+1}, \ldots, x_n)$, defined as the sum of all monomials (not necessarily squarefree) of degree $d$ in $x_{r'},\ldots,x_n$.  We prove a number of identities involving elementary and complete symmetric functions in Section~\ref{sec:alg_section}, and a remarkable identity relating truncated elementary symmetric functions to truncated complete symmetric functions in Proposition~\ref{prop:crucial_identity}.

Our motivation comes from algebraic geometry and representation theory: we generalize an important family of ideals of symmetric functions called {\em Tanisaki ideals}.  Springer theory is a key example of geometric representation theory that constructs representations of the symmetric group $S_n$ on the cohomology of a family of varieties parametrized by partitions~\cite{Spr}.  Kraft conjectured that the cohomology ring of the Springer variety for the partition $\lambda$ was a particular quotient $\mathbb{Z}[x_1,\ldots,x_n]/\mathcal{I}_{\lambda}$ of the polynomial ring~\cite{Kra80}.  De Concini and Procesi proved Kraft's conjecture by constructing the ideals $\mathcal{I}_\lambda$ explicitly~\cite{dCP}.  Tanisaki simplified the ideals $\mathcal{I}_{\lambda}$ that appear in this quotient~\cite{Tani}; they are now called Tanisaki ideals.  Unlike earlier constructions of Springer's representations, the $S_n$-action is transparent in this presentation: it is simply the natural $S_n$-action on the variables of $\mathbb{Z}[x_1, \ldots, x_n]$, extended to the quotient $\mathbb{Z}[x_1,\ldots,x_n]/\mathcal{I}_{\lambda}$ because $\mathcal{I}_{\lambda}$ is symmetric.  

Springer varieties can be generalized to a two-parameter family of varieties called {\em Hessenberg varieties}~\cite{dMPS}, defined by a partition $\lambda$ and a certain step function $h$ (or equivalently, a Dyck path).  Examples of Hessenberg varieties have been studied in different contexts: quantum cohomology~\cite{Kos96}, \cite{Rie03}, combinatorics~\cite{Ful99}, \cite{Ive06}, geometry~\cite{BriCar04}, and topology~\cite{Tym06}, \cite{HarTym}.  Yet remarkably little is known about them: for instance, outside of special cases like the Springer varieties or the Peterson variety~\cite{BriCar04}, \cite{HarTym}, \cite{HarTym2}, their cohomology ring is unknown.

In this paper, we extend the algebraic and combinatorial approach of Kraft, De Concini-Procesi, and Tanisaki from the Springer varieties to the {\em regular nilpotent Hessenberg varieties}, namely the Hessenberg varieties corresponding to the partition $(n)$.  Regular nilpotent Hessenberg varieties have been studied extensively~\cite{Kos96}, \cite{Rie03}, \cite{BriCar04}, \cite{Tym07}, \cite{HarTym}.  Section~\ref{sec:I_ideals} uses truncated elementary symmetric functions to build a family of ideals $I_h$ that are parametrized by $h$ and that correspond to regular nilpotent Hessenberg varieties in a sense made precise later.  We call each $I_h$ a  \textit{generalized Tanisaki ideal} because we modeled the construction of $I_h$ on Biagioli, Faridi, and Rosas's construction of the Tanisaki ideal, in the case when the partition $\lambda = (n)$~\cite{Biag08}.   Unlike the Tanisaki ideals, $I_h$ are not symmetric and so $\mathbb{Z}[x_1, x_2, \ldots, x_n]/I_h$ does not carry an obvious $S_n$-action.  However, we prove a number of powerful properties satisfied by the ideals $I_h$:
\begin{enumerate}
\item If $h > h'$ is a natural partial order on step functions, then $I_h \subset I_{h'}$ (see Theorem~\ref{thm:I_ideal_containment}).  A stronger condition sometimes holds: in some cases, the generators $\mathfrak{C}_h$ of $I_h$ are actually a subset of the generators $\mathfrak{C}_{h'}$ of $I_{h'}$.  We prove that the set of pairs $h > h'$ for which the generators satisfy $\mathfrak{C}_h \subset \mathfrak{C}_{h'}$ induces a spanning subgraph of the poset on Hessenberg functions $h$ (see Corollary~\ref{cor:All_Paths_Lead_To_Rome}).
\item We identify a reduced generating set for the ideal $I_h$ consisting of $n$ truncated elementary symmetric functions (see Theorem~\ref{thm:min_gen_set_for_I}).  Galetto recently proved that this reduced generating set is in fact minimal~\cite{Galetto}.  We give his proof in Appendix~\ref{app:minimal_gen_proof}.
\item We identify a Gr\"{o}bner basis for $I_h$ consisting of $n$ truncated complete symmetric functions, with respect to two term orders (see Theorem~\ref{thm:J_ideal_Grob}).  This generalizes the case when $h=(n,n,\ldots,n)$, namely when $\mathbb{Q}[x_1,\ldots,x_n]/I_h$ is the cohomology of the full flag variety, which is due in various incarnations to Cauchy, Valibouze, and Mora-Sala~\cite{AC},\cite{AV},\cite{MS}.
\item We identify a monomial basis for $\mathbb{Q}[x_1,\ldots,x_n]/I_h$ (see Theorem~\ref{thm:Basis_for_R/J}).  
\end{enumerate}
To prove points (3) and (4), we construct an entirely new family of ideals $J_h$ that we call \textit{Hessenberg basis ideals}.  We then prove in Section~\ref{I=J} the main theorem of this paper: Theorem~\ref{thm:I_h=J_h}, which says that $I_h=J_h$ for each $h$.  

In earlier work, the first author proved that the quotient $\mathbb{Q}[x_1,\ldots,x_n]/J_h$ is isomorphic as a $\mathbb{Q}$-vector space to the cohomology of the regular nilpotent Hessenberg variety for $h$~\cite{Mb1}.  In this sense the ideals $I_h=J_h$ generalize the Tanisaki ideals {\em geometrically} as well as combinatorially.

Even earlier, Ding, Gasharov-Reiner, and Develin-Martin-Reiner studied ideals $J_h'$ that agree with $J_h$ up to the change of variable $x_i \leftrightarrow x_{n-i+1}$.  Ding proved that the cohomology ring  of  {\em partition varieties}, a family of smooth Schubert varieties parametrized by $h$, was isomorphic as an additive group to $\mathbb{Z}[x_1,\ldots,x_n]/J_h'$~\cite{Din}.   Gasharov-Reiner proved that in fact the cohomology was ring-isomorphic to $\mathbb{Z}[x_1,\ldots,x_n]/J_h'$~\cite{GasRei}, \cite{DevMarRei}.  (The change of variables corresponds to a natural homeomorphism of the flag variety to itself that sends partition varieties to another family of smooth Schubert varieties, whose cohomology rings are exactly $\mathbb{Z}[x_1,\ldots,x_n]/J_h$.) While this article was going to press, we learned that in fact Gasharov-Reiner constructed the same Gr\"{o}bner basis that we give here for a family of ideals that includes $J_h'$ as a special case~\cite[Remark 3.3]{GasRei}.  

This leads to the natural conjecture that regular nilpotent Hessenberg varieties are partition varieties---natural, yet false, since recent work of Insko and Yong identifies the singular locus of many regular nilpotent Hessenberg varieties~\cite{InsYon}.  In particular, Hessenberg varieties are typically singular while partition varieties are always smooth.  However, we conjecture that the cohomology of the regular nilpotent Hessenberg variety is ring-isomorphic to the cohomology of the corresponding partition variety.

Throughout this paper, our methods are purely algebraic and combinatorial.


\section{Combinatorial preliminaries}\label{sec:prelim_definitions}

In this section, we define the key combinatorial objects of this paper: a family of nondecreasing step functions called {\em Hessenberg functions}.  Proposition~\ref{prop:one_to_one_correspondences} establishes bijections between Hessenberg functions and several other combinatorial objects, both classical (like Dyck paths) and not.  Section~\ref{sec:Hess_poset} describes a natural poset on Hessenberg functions together with its basic properties.  We describe our geometric motivation in Section~\ref{section:geom_prelims}.

\subsection{Hessenberg functions and degree tuples}  We fix a positive integer $n$ and the polynomial ring $\mathbb{Z}[x_1,\ldots,x_n]$ once and for all.  We identify maps from $\{1,2,\ldots,n\} \rightarrow \mathbb{Z}$ with elements of $\mathbb{Z}^n$, as in the next definition.

\begin{definition}[Hessenberg function]\label{def:Hess_fcn}
Let $h$ be a map 
\[h: \{1,\ldots,n\} \rightarrow \{1,\ldots,n\}\] 
and let $h_i=h(i)$ denote the image of $i$ under $h$.  We say that an $n$-tuple $h=(h_1,\ldots,h_n)$ is a \textit{Hessenberg function} if it satisfies two structure rules:
\[\begin{array}{lll}
(a) & i \leq h_i \leq n  & \textup{  for all }i\in\{1,\ldots,n\}, \textup{ and}\\
(b) & h_i \leq h_{i+1}  & \textup{  for all }i\in\{1,\ldots,n-1\}.
\end{array}\]
\end{definition}

We will show that the set of Hessenberg functions are in natural bijection with the well-known {\em Dyck paths}, as well as with combinatorial objects called {\em degree tuples} that we introduce below.  We recall the definition of a Dyck path (rotating standard conventions $90^{\circ}$ clockwise).

\begin{definition}[Dyck path] Consider the square $\{(x,y): 0 \leq x,y \leq n\}$ in the plane.  A \textit{Dyck path} is any path in the square from $(0,n)$ to $(n,0)$ that 
\begin{enumerate}
\item[(a)] lies strictly below the antidiagonal $y=-x+n$, and 
\item[(b)] is a lattice path, i.e., consists of vertical and horizontal segments of length one.
\end{enumerate}
\end{definition}

We now define degree tuples, named because they index the degrees of certain truncated symmetric functions constructed in Section~\ref{sec:J_ideals}.  Note that the entries of the degree tuple are listed from $\beta_n$ to $\beta_1$.  (This convention will be convenient in Section~\ref{sec:J_ideals}.)

\begin{definition}[Degree tuple]\label{def:deg_tuples}
Let $\beta$ be a map $\beta: \{1,\ldots,n\} \rightarrow \{1,\ldots,n\}$ and denote the image of $i$ under $\beta$ by $\beta_i=\beta(i)$.  We say that an $n$-tuple $\beta=(\beta_n,\beta_{n-1},\ldots,\beta_1)$ is a \textit{degree tuple} if it satisfies two structure rules:
\[\begin{array}{lll}
(a') & 1 \leq \beta_i \leq i & \textup{  for all }i\in\{1,\ldots,n\}, \textup{ and}\\
(b') & \beta_i - \beta_{i-1} \leq 1 & \textup{  for all }i\in\{2,\ldots,n\}.
\end{array}\]
\end{definition}

For us, a {\em partition} $\lambda = (\lambda_1,\ldots,\lambda_n)$ satisfies $n \geq \lambda_1\geq\ldots\geq\lambda_n\geq0$.  We draw Ferrers diagrams flush right and top.  For example if $n=3$, then:
\[\lambda=(3,1,0) \longleftrightarrow
\setlength{\unitlength}{.15in}
\begin{picture}(3,3)(0,0.75)
\linethickness{.2pt}
\multiput(0,0)(1,0){4}{\line(0,1){3}}
\multiput(0,0)(0,1){4}{\line(1,0){3}}
\multiput(.5,2.5)(1,0){3}{\makebox(0,0){$\blacksquare$}}
\put(2.5,1.5){\makebox(0,0){$\blacksquare$}}
\end{picture}
\]

\medskip

We remind the reader of some standard definitions involving partitions.

\begin{definition}[Staircase partition, ample partition, conjugate of a partition] The {\em staircase partition} is defined to be $\rho=(n,n-1,\ldots,1)$.  The partition $\lambda = (\lambda_1,\ldots,\lambda_n)$ is \textit{ample} if  $\rho\subseteq\lambda$, that is, if $\lambda_i \geq n-i+1$ for each $i$.  If $\lambda = (\lambda_1,\ldots,\lambda_n)$ is a partition, the \textit{conjugate} $\lambda'$ of $\lambda$ is given by $\lambda_i' = \#\{k:\lambda_k\geq i\}$.  
\end{definition}

With our convention for Ferrers diagrams, the conjugate $\lambda'$ is the reflection of $\lambda$ across the antidiagonal line.  The next lemma follows from this characterization.

\begin{lemma}\label{lem:ample_partition_closure} The set of ample partitions is closed under conjugation.
\end{lemma}

The next definition applies to partitions, Hessenberg functions, and degree tuples.

\begin{definition}[Reverse tuple]  If $t=(t_1,\ldots,t_n) \in \mathbb{Z}^n$ then the \textit{reverse of $t$} is
\[t^{rev} = (t_n,t_{n-1},\ldots,t_1).\]
\end{definition}

The main result of Section~\ref{sec:prelim_definitions} follows; it proves bijections between these objects.

\begin{proposition}\label{prop:one_to_one_correspondences} The following sets are in bijective correspondence:
$$\begin{array}{ll}
\mbox{(1)} & \mbox{The set of ample partitions.}\\
\mbox{(2)} & \mbox{The set of Hessenberg functions.}\\
\mbox{(3)} & \mbox{The set of degree tuples.}\\
\mbox{(4)} & \mbox{The set of Dyck paths.}
\end{array}$$
\end{proposition}

\begin{proof}
We prove that each of sets (2), (3), and (4) is in bijection with ample partitions.  Let $h=(h_1,\ldots,h_n) \in \{1,2,\ldots,n\}^n$ and let $h^{rev}=(h_n,h_{n-1},\ldots,h_1)$ be the reverse of $h$.  The map $h$ satisfies rule ($b$) in the definition of Hessenberg functions exactly when $h^{rev}$ is a partition, and satisfies rule ($a$) precisely when this partition $h^{rev}$ is ample.

Let $\lambda = (\lambda_1,\ldots,\lambda_n) \in \{1,2,\ldots,n\}^n$ and define $\beta=(\beta_n,\beta_{n-1},\ldots,\beta_1)$ by 
\[\beta_i = \lambda_i-\rho_i+1 = \lambda_i+i-n.\]  
Observe that $\beta$ satisfies rule ($b'$) in the definition of degree tuples exactly when $\lambda$ is a partition, and satisfies rule ($a'$) precisely when this partition $\lambda$ is ample.

Finally, the \textit{boundary path} of a partition is the path between the partition and its complement in the $n$-by-$n$ square.  A partition is determined by its boundary path.  By definition a boundary path is a Dyck path precisely when its corresponding partition is ample.
\end{proof}

Using Lemma~\ref{lem:ample_partition_closure} and Proposition~\ref{prop:one_to_one_correspondences} we can define a bijective map from Hessenberg functions to degree tuples.  Let $1$ denote the partition $(1,\ldots,1)$ and let $\rho$ be the staircase partition, as usual.  We define a composition of bijections $F$ as follows:
$$F:h \longmapsto h^{rev} \longmapsto \left(h^{rev}\right)' \longmapsto \left(h^{rev}\right)'-\rho+1 \longmapsto \left(\left(h^{rev}\right)'-\rho+1\right)^{rev}.$$
The map $F$ takes a Hessenberg function to a degree tuple.  

\begin{corollary}\label{cor:hess_fcns_biject_with_degree_tuples}
The map $F$ is a bijection between Hessenberg functions and degree tuples.
\end{corollary}
\begin{proof}
Each map used to define $F$ is bijective and the composition sends Hessenberg functions to degree tuples, both by Proposition~\ref{prop:one_to_one_correspondences}. The last map 
\[\left(h^{rev}\right)'-\rho+1 \mapsto \left(\left(h^{rev}\right)'-\rho+1\right)^{rev}\] 
ensures that the degree tuple has descending subscripts, as in Definition~\ref{def:deg_tuples}.
\end{proof}

It is often helpful to represent $F(h)$ with a diagram, which we use extensively in Section~\ref{sec:Ideal_equality_section}.  The diagram relies on two observations.  First, the map $h \mapsto (h^{rev})'$ produces a partition whose $i^{th}$ column has length $h_i$.  Second, the map $(h^{rev})' \mapsto (h^{rev})'-\rho+1$ essentially erases the strictly upper-triangular portion of the partition.

\begin{definition}[Hessenberg diagram]\label{def:Hess_diagram}
Let $h$ be a Hessenberg function.  For each $i$, shade the first $h_i$ boxes in the $i^{th}$ column of the Ferrers diagram.  Remove the boxes in the partition $(n-1,n-2,\ldots,1,0)$.  The diagram that remains is the \textit{Hessenberg diagram} of $h$.
\end{definition}

\begin{example}
The Hessenberg diagram of $h=(3,3,4,4,5,6)$, with its Dyck path, is:

\medskip
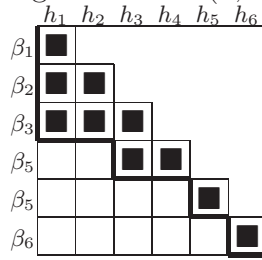
\begin{center}\setlength{\unitlength}{.2in}
\begin{picture}(6,6)(0,0)
\linethickness{.2pt}
\put(0,0){\line(1,0){6}}
\put(0,1){\line(1,0){6}}
\put(0,2){\line(1,0){5}}
\put(0,3){\line(1,0){4}}
\put(0,4){\line(1,0){3}}
\put(0,5){\line(1,0){2}}
\put(0,6){\line(1,0){1}}
\put(0,6){\line(1,0){6}}
\put(6,6){\line(0,-1){6}}
\put(0,0){\line(0,1){6}}
\put(1,6){\line(0,-1){6}}
\put(2,5){\line(0,-1){5}}
\put(3,4){\line(0,-1){4}}
\put(4,3){\line(0,-1){3}}
\put(5,2){\line(0,-1){2}}
\multiput(.5,5.5)(0,-1){3}{\makebox(0,0){$\blacksquare$}}
\multiput(1.5,4.5)(0,-1){2}{\makebox(0,0){$\blacksquare$}}
\multiput(2.5,3.5)(0,-1){2}{\makebox(0,0){$\blacksquare$}}
\multiput(3.5,2.5)(0,-1){1}{\makebox(0,0){$\blacksquare$}}
\multiput(4.5,1.5)(0,-1){1}{\makebox(0,0){$\blacksquare$}}
\multiput(5.5,.5)(0,-1){1}{\makebox(0,0){$\blacksquare$}}
\linethickness{1.5pt}
\put(0,6){\line(0,-1){3}}
\put(0,3){\line(1,0){2}}
\put(2,3){\line(0,-1){1}}
\put(2,2){\line(1,0){2}}
\put(4,2){\line(0,-1){1}}
\put(4,1){\line(1,0){1}}
\put(5,1){\line(0,-1){1}}
\put(5,0){\line(1,0){1}}
\linethickness{.2pt}	
\put(.15,6.1){\begin{scriptsize}$h_1$\end{scriptsize}}\put(1.15,6.1){\begin{scriptsize}$h_2$\end{scriptsize}}
\put(2.15,6.1){\begin{scriptsize}$h_3$\end{scriptsize}}\put(3.15,6.1){\begin{scriptsize}$h_4$\end{scriptsize}}
\put(4.15,6.1){\begin{scriptsize}$h_5$\end{scriptsize}}\put(5.15,6.1){\begin{scriptsize}$h_6$\end{scriptsize}}
\put(-.7,.25){\begin{scriptsize}$\beta_6$\end{scriptsize}}\put(-.7,1.25){\begin{scriptsize}$\beta_5$\end{scriptsize}}
\put(-.7,2.25){\begin{scriptsize}$\beta_5$\end{scriptsize}}\put(-.7,3.25){\begin{scriptsize}$\beta_3$\end{scriptsize}}
\put(-.7,4.25){\begin{scriptsize}$\beta_2$\end{scriptsize}}\put(-.7,5.25){\begin{scriptsize}$\beta_1$\end{scriptsize}}
\end{picture}
\end{center}
\end{example}

In the previous example, there are $h_i$ boxes in the $i^{th}$ column between the Dyck path and the top of the square, and $\beta_i$ shaded boxes in row $i$.  This leads to the following observation.

\begin{lemma}[Formula to compute $F(h)$ from $h$]\label{lemma:simple_formula_to_compute_beta}  
If $h=(h_1,\ldots,h_n)$ is a Hessenberg function, then the degree tuple $\beta = F(h)$ is the sequence $\beta = (\beta_n,\beta_{n-1},\ldots,\beta_1)$ where
\[\beta_i = i - \#\{h_k | h_k < i\}.\]
\end{lemma}

\begin{proof}
By definition $\beta_i =  \#\{h_{n-k+1} | h_{n-k+1} \geq i\} + i-n$.  Reindexing the set, we obtain $\beta_i =  \#\{h_{k} | h_{k} \geq i\} + i-n$, which is the number of shaded boxes below the diagonal on the $i^{th}$ row of the Hessenberg diagram of $h$.  If we subtract unshaded boxes below the diagonal rather than shaded boxes above the diagonal, we get $\beta_i = i - \#\{h_k | h_k < i\}$, as desired.
\end{proof}

\subsection{Posets on Hessenberg functions and degree tuples}\label{sec:Hess_poset}

Hessenberg functions and degree tuples have natural partial orders induced from partial orders on the lattice $\mathbb{Z}^n$.  We use the following, which corresponds to the partial ordering on Dyck paths by containment.

\begin{definition}[Poset on Hessenberg functions]\label{def:h_sequence_tree}
Consider the two Hessenberg functions $h=(h_1,\ldots,h_n)$ and $h'=(h'_1,\ldots,h'_n)$.  The {\em partial order on Hessenberg functions} is defined by the rule that $h \leq h'$ if and only if $h_i \leq h_i'$ for all $i$.  The \textit{Hasse diagram on Hessenberg functions} is the directed graph whose vertices are Hessenberg functions, and with an edge from $h$ to $h'$ if exactly one entry in $h'$ is one less than its corresponding entry in $h$, in other words if $h'_{i_0} = h_{i_0} - 1$ for some $i_0$ but $h'_i = h_i$ for all $i\neq i_0$.   
\end{definition}

Recall that $h \geq h'$ in the partial order if and only if there is a path from $h$ to $h'$ in the Hasse diagram.  The left side of Figure~\ref{fig:h_beta_trees} gives an example of the Hasse diagram on Hessenberg functions when $n=4$. For all $n$, the top vertex of the Hasse diagram is the function $(n,\ldots,n)$ and the bottom vertex is $(1,2,\ldots,n)$.  (The double-lined dashed edges in Figure~\ref{fig:h_beta_trees} will be explained in Section~\ref{subsec:special_containment_of_I_ideals}.)
 
We define a partial order on degree tuples similarly, with $\beta \geq \beta'$ if $\beta_i \geq \beta_i'$ for all $i$.  The right side of Figure~\ref{fig:h_beta_trees} gives an example of the Hasse diagram when $n=4$.  The reader may observe that the Hasse diagram for this partial order is the same as that for Hessenberg functions.  By Lemma~\ref{lemma:simple_formula_to_compute_beta}, the map $F$ from Hessenberg functions to degree tuples that was defined in Corollary~\ref{cor:hess_fcns_biject_with_degree_tuples} preserves these partial orders.  

\begin{figure}[t]
\begin{center}
\subfigure[Hessenberg functions.]{
\xymatrix{
& h=4444 \ar@{:>}[d] \\
& 3444 \ar@{:>}[dr] \ar[dl] \\
2444 \ar[d] \ar@{:>}[dr] &  & 3344 \ar@{:>}[d] \ar@{:>}[dl] \\
1444 \ar@{:>}[d] & 2344 \ar[dl] \ar[d] \ar@{:>}[dr] & 3334\ar@{:>}[d] \\
1344 \ar@{:>}[dr] \ar[d] & 2244 \ar@{:>}[dl] \ar@{:>}[dr] & 2334 \ar@{:>}[d] \ar[dl] \\
1244 \ar@{:>}[dr] & 1334 \ar@{:>}[d] & 2234 \ar@{:>}[dl] \\
& 1234
}
}
\hspace{.25in}
\subfigure[Degree tuples.]{
\xymatrix{
& \beta=4321 \ar[d] \\
& 3321 \ar[dl] \ar[dr] \\
3221 \ar[d] \ar[dr] &  & 2321 \ar[d] \ar[dl] \\
3211 \ar[d] & 2221 \ar[d] \ar[dr] \ar[dl] & 1321\ar[d] \\
2211 \ar[d] \ar[dr] & 2121 \ar[dl] \ar[dr] & 1221 \ar[d] \ar[dl] \\
2111 \ar[dr] & 1211 \ar[d] & 1121 \ar[dl] \\
& 1111
}
}
\caption{\label{fig:h_beta_trees} Hasse diagrams for $n=4$.}
\end{center}
\end{figure}
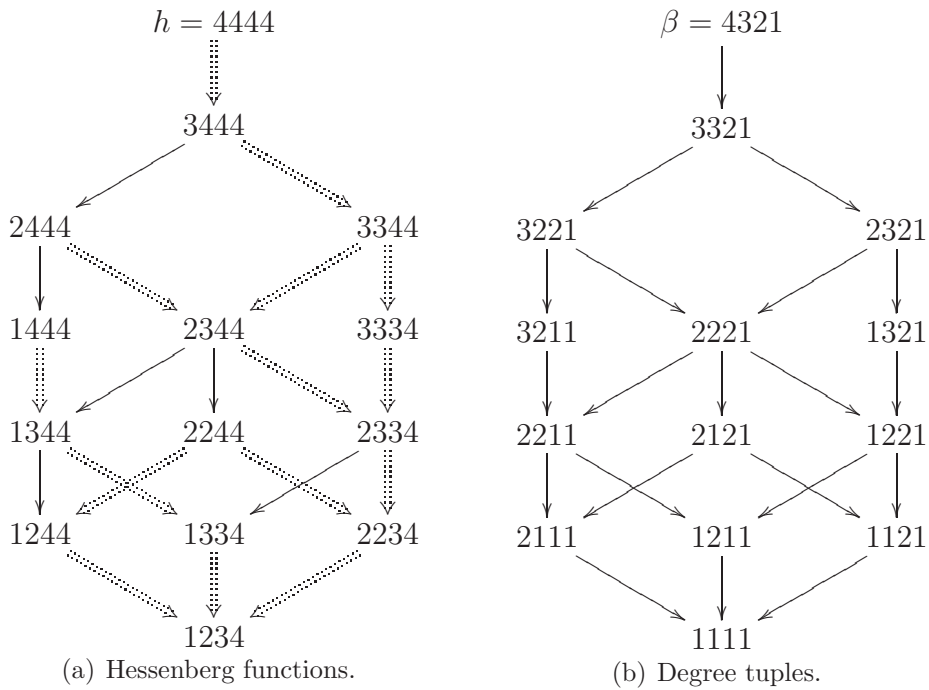

The next corollary fits Hessenberg functions and degree tuples into the existing literature, including Stanley's list of the (many) combinatorial interpretations of Catalan numbers~\cite{RS2}.  It follows immediately from Proposition~\ref{prop:one_to_one_correspondences}, together with the corresponding results counting Dyck paths or chains of Dyck paths, respectively.

\begin{corollary}\label{cor:CatalanPaths} The following two enumerative properties hold:

\begin{enumerate}
\item The number of Hessenberg functions (equivalently, degree tuples) is
$$ \mbox{Catalan(n)} = \frac{1}{n+1} \left(\begin{array}{c} {2n} \\ {n} \end{array}\right).$$
\item The number of maximal chains, namely chains from the top vertex down to the bottom vertex, on the Hasse diagram of Hessenberg functions (equivalently, degree tuples) is
$$\frac{\left(\begin{array}{c} {n} \\ {2} \end{array}\right)!}{\prod_{i=1}^{n-1} (2i-1)^{n-1}}.$$
\end{enumerate}
\end{corollary}

\begin{proof}
The first claim follows from Stanley's enumeration of Dyck paths~\cite[Corollary 6.2.3(v)]{RS2}.  Woodcock has a particularly ingenious proof of the second claim~\cite[Proposition 50]{JW}; Stanley gives a different version of this formula~\cite[pg.116]{RS1}. 
\end{proof}

\begin{example}
Figure~\ref{fig:h_beta_trees} shows there are exactly 14 nodes, which is the $4^{th}$ Catalan number by Corollary~\ref{cor:CatalanPaths}.  The number of maximal chains of Hessenberg functions is
$$\frac{\left(\begin{array}{c} {4} \\ {2} \end{array}\right)!}{\prod_{i=1}^{4-1} (2i-1)^{4-1}} \; = \; 16.$$
\end{example}

\subsection{Geometric interpretations} \label{section:geom_prelims}
We study Hessenberg functions for the following geometric application.  A \textit{flag} in $\mathbb{C}^n$ is a nested sequence of vector subspaces
$$V_1 \subseteq V_2 \subseteq \cdots \subseteq V_n = \mathbb{C}^n$$
where each $V_i$ has dimension $i$.  The collection of all such flags is called the \textit{full flag variety} and is denoted $\mathfrak{F}$.  Hessenberg varieties are parametrized by a linear operator together with a Hessenberg function, as follows.

\begin{definition}[Hessenberg variety]
Fix a nilpotent matrix $X \in \textup{Mat}_n(\mathbb{C})$ and let $h$ be a Hessenberg function.  The \textit{Hessenberg variety} of $X$ and $h$ is the following subvariety:
$$\mathfrak{H}(X,h) = \{\mbox{Flags}\in\mathfrak{F} \;|\; X\cdot V_i\subseteq V_{h(i)} \mbox{ for all $i$} \}.$$
\end{definition}

For instance, when $h=(n,\ldots,n)$ then every flag satisfies  $X\cdot V_i\subseteq V_{h(i)}$ and so $\mathfrak{H}(X,h)$ is the full flag variety.  When $h=(1,\ldots,n)$ or equivalently $h(i)=i$ for all $i$, then $\mathfrak{H}(X,h)$ is the {\em Springer variety} from the Introduction.


\section{Decomposition identities on truncated symmetric functions}\label{sec:alg_section}

\subsection{Algebraic preliminaries}
We will study polynomials that are symmetric in a subset of the variables $\{x_1,\ldots, x_n\}$, focusing particularly on elementary and complete symmetric functions.  While not symmetric functions, these {\em truncated symmetric functions} retain significant combinatorial structure.  In this section, we prove a series of identities involving truncated symmetric functions.  

\begin{definition}[Truncated symmetric functions]\label{def:truncated_sym_fcn}
For $S \subseteq \{1,\ldots,n\}$ and $d>0$, the  \textit{truncated elementary symmetric function} $e_d(S)$ is defined to be the following sum:
$$e_d(S) = \sum_{\{ i_1<\cdots<i_d\} \subseteq S} x_{i_1} x_{i_2} \cdots x_{i_d}.$$
The  \textit{truncated complete symmetric function} $\tilde{e}_d(S)$ is defined as the following sum of (not necessarily squarefree) monomials:
$$\tilde{e}_d(S) = \sum_{\begin{array}{c} \textup{multisets} \\
\{i_1\leq\cdots\leq i_d\}  \subseteq S\end{array}}
x_{i_1} x_{i_2} \cdots x_{i_d}.$$
Our convention is that if $d=0$ then $e_d(S) = \tilde{e}_d(S) = 1$, including when $S$ is the empty set.  If $d<0$ and $S \neq \emptyset$ then $e_d(S) = \tilde{e}_d(S) = 0$.  If $d>|S|$ then $e_d(S)=0$.
\end{definition}

\begin{example}
Fix $n=4$.  If $d=2$ and $r=3$, then the truncated symmetric function $e_2(1,2,3) = x_1x_2+x_1x_3+x_2x_3$ and $\tilde{e}_2(3,4) = x_3^2+x_3x_4+x_4^2$.
\end{example}

\begin{remark}
Although some other sets $S$ will arise naturally in our calculations, for notational convenience we will typically use either the set $S=\{1,2,\ldots,r\}$ or the set $S = \{r,r+1,\ldots,n\}$.  When $r=n$, the function $e_d(1,\ldots,n)$ is the elementary symmetric function of degree $d$, denoted $e_d$ in this paper like elsewhere in the literature.  When $r=1$, the function $\tilde{e}_d(1,\ldots,n)$ is the complete symmetric function of degree $d$, which is usually denoted $h_d$.  We avoid standard notation in this case because $h$ is used elsewhere.
\end{remark}

\subsection{Four decomposition identities} This section collects several identities on truncated symmetric functions.  We use these identities repeatedly in the remaining sections.  The first two, which are fundamental to our analysis in Sections~\ref{subsec:general_containment_of_I_ideals} and~\ref{sec:Ideal_equality_section}, express the truncated elementary symmetric function $e_d(1,\ldots,r)$ in terms of $e_{d-1}(1,\ldots,r-1)$ and $e_d(1,\ldots,r-1)$, and similarly for the truncated complete symmetric functions.  Both are well known (e.g.,~\cite[pg.~21]{Pro07}, \cite[Equation (3.1)]{FGP97}).

\begin{lemma}\label{lem:decomposing_truncated_elem_symfcn}
The truncated elementary symmetric function $e_d(1,\ldots,r)$ decomposes as
$$e_d(1,\ldots,r) = x_r \cdot e_{d-1}(1,\ldots,r-1) + e_d(1,\ldots,r-1)$$
for $d < r$.  If $d=r$, then $e_d(1,\ldots,r)=x_r \cdot e_{d-1}(1,\ldots,r-1)$.
\end{lemma}


\begin{lemma}\label{lem:decomposing_truncated_complete_symfcn}
The truncated complete symmetric function $\tilde{e}_d(r,r+1,\ldots,n)$ decomposes as
$$\tilde{e}_d(r,r+1,\ldots,n) = x_r \cdot \tilde{e}_{d-1}(r,r+1,\ldots,n) + \tilde{e}_d(r+1,r+2,\ldots,n)$$
for $r<n$.  If $r=n$, then $\tilde{e}_d(n) = x_n \cdot \tilde{e}_{d-1}(n)$, where $\tilde{e}_i(n)  = x_n^i$.
\end{lemma}


The next lemma will decompose the truncated elementary symmetric function $e_d(1,\ldots,r) $ as a linear combination of truncated elementary symmetric functions with a {\em fixed}, smaller variable set but \emph{varying} degrees.

\begin{lemma}\label{lem:cool_combinatorial_sym_func_lemma}
Let $d,r,n$ be positive integers such that $d<r\leq n$ and fix $j<r$.  The  function $e_d(1,\ldots,r)$ is a $\mathbb{Z}[x_1,\ldots,x_n]$-linear combination of truncated elementary symmetric functions in the variables $x_1,\ldots,x_{r-j}$:
\begin{equation}
e_d(1,\ldots,r) = \sum_{t=0}^j e_t(r-j+1,r-j+2,\ldots,r) \cdot e_{d-t}(1,\ldots,r-j).\label{equ:cool_sum_for_lemma}
\end{equation}
\end{lemma}

\begin{proof}
By definition, the function $e_d(1,\ldots,r) = \sum x_{i_1} x_{i_2} \cdots x_{i_d}$ where the sum is taken over the $\left(\begin{array}{c} {r} \\ {d} \end{array}\right)$ different subsets $\{i_1 < \cdots < i_d\}$ of $\{1,2,\ldots,r\}$.  We describe an alternate strategy to enumerate the subsets of $\{1,2,\ldots,r\}$.  For each $t=0,1,\ldots,j$, choose $t$ elements from the set $\{r-j+1,r-j+2,\ldots,r\}$ and choose $d-t$ elements independently from $\{1,2,\ldots,r-j\}$. By definition, for each fixed $t \in \{0,1,2,\ldots,j\}$, the product 
\begin{align}
e_t(r-j+1,r-j+2,\ldots,r) \cdot e_{d-t}(1,\ldots,r-j) 
\label{equ:summand_in_cool_lemma}
\end{align}
sums exactly those monomials $x_{i_1} \cdots x_{i_t} x_{i_{t+1}} \cdots x_{i_d}$ whose subscripts satisfy:
$$r-j+1 \leq i_1 < \cdots < i_t \leq r \;\; \mbox{ and } \;\; 1 \leq i_{t+1} < \cdots < i_d \leq r-j.$$  
We conclude $\sum_{t=0}^j e_t(r-j+1,r-j+2,\ldots,r) \cdot e_{d-t}(1,\ldots,r-j) = e_d(1,\ldots,r)$ as desired.
\end{proof}

\begin{remark}\label{rem:nonzero_terms_in_cool_sym_func_lemma}
If $j>r$ then all of the summands in Equation~\eqref{equ:summand_in_cool_lemma} are zero, so we assume $0\leq j\leq r$.  In addition, the product in Equation~\eqref{equ:summand_in_cool_lemma} may be zero for certain values of $t$.  The second factor is zero unless $d-t \leq r-j$.  (The first factor is always nonzero since it has $j$ variables and degree $t\leq j$.)  Hence the product in Equation~\eqref{equ:summand_in_cool_lemma} is nonzero whenever $t \geq max\{0,d-r+j\}$.  
\end{remark}

\begin{remark}
By setting all $x_i=1$ and counting terms in Lemma~\ref{lem:cool_combinatorial_sym_func_lemma}, we obtain another proof of the well-known combinatorial identity:
$$\sum_{t=0}^j \binom{j}{t} \cdot \binom{r-j}{d-t} = \binom{r}{d}.$$
\end{remark}

The final identity in this section is an analogue of Lemma~\ref{lem:cool_combinatorial_sym_func_lemma} for complete symmetric functions.  It will decompose the truncated complete symmetric function $\tilde{e}_d(r,r+1,\ldots,n)$ as a linear combination of truncated complete symmetric functions all of which have the \emph{same} smaller variable set $x_{r+1},x_{r+2},\ldots,x_n$ but \emph{varying} degrees.  Unlike Lemma~\ref{lem:cool_combinatorial_sym_func_lemma}, we cannot completely eliminate an expression $\tilde{e}_{d'}(r,r+1,\ldots,n)$ involving the original variable set.  

\begin{lemma}\label{lem:complete_symm_func_lemma}
Fix any $d,d'$ with $0 \leq d' \leq d \leq n$.  Then
\[\tilde{e}_d(r,\ldots,n) = x^{d-d'}_r \cdot \tilde{e}_{d'}(r,\ldots,n) + \sum_{t=1}^{d-d'} x^{d-(d'+t)}_r \cdot \tilde{e}_{d'+t}(r+1,\ldots,n).\]
\end{lemma}

\begin{proof}
We induct on the difference $d-d'$.  The claim is vacuously true if $d-d'=0$.  We assume the claim holds for the pair $d, d'+1$, whose difference is $d-d' = i-1$.  In other words, we assume that $\tilde{e}_d(r,\ldots,n)$ decomposes as
\vspace{-.1in}
\begin{align}
x^{d-(d'+1)}_r \cdot \tilde{e}_{d'+1}(r,\ldots,n) + \sum_{t=1}^{d-(d'+1)} x^{d-(d'+1+t)}_r \cdot \tilde{e}_{(d'+1)+t}(r+1,\ldots,n).\label{eqn:induction hyp for complete lemma}
\end{align}
We prove the claim holds for the pair $d,d'$ with $d-d'=i$.  Lemma~\ref{lem:decomposing_truncated_complete_symfcn} implies 
$$\tilde{e}_{d'+1}(r,\ldots,n) = x_r \cdot \tilde{e}_{d'}(r,\ldots,n) + \tilde{e}_{d'+1}(r+1,\ldots,n).$$
Substituting into Equation~\eqref{eqn:induction hyp for complete lemma} and then incorporating into the sum, we obtain
\begin{align*}
\tilde{e}_d(r,\ldots,n) &= x^{d-(d'+1)}_r \cdot \left( x_r \cdot \tilde{e}_{d'}(r,\ldots,n) + \tilde{e}_{d'+1}(r+1,\ldots,n) \right) \\
				&\hspace{.75in} + \; \sum_{t=1}^{d-(d'+1)} x^{d-(d'+1+t)}_r \cdot \tilde{e}_{(d'+1)+t}(r+1,\ldots,n) \\
				&= x^{d-d'}_r  \cdot \tilde{e}_{d'}(r,r+1,\ldots,n) \\
				&\hspace{.75in} + \; \sum_{t=0}^{d-(d'+1)} x^{d-(d'+1+t)}_r \cdot \tilde{e}_{(d'+1)+t}(r+1,\ldots,n).
\end{align*}
Reparametrizing $t$ gives
$$\tilde{e}_d(r,\ldots,n) = x^{d-d'}_r \cdot \tilde{e}_{d'}(r,\ldots,n) +  \sum_{t=1}^{d-d'} x^{d-(d'+t)}_r \cdot \tilde{e}_{d'+t}(r+1,\ldots,n)$$
as desired.  By induction, the claim holds for all $0 \leq d' \leq d \leq n$.
\end{proof}


\section{Family of generalized Tanisaki ideals \texorpdfstring{$I_h$}{I}}\label{sec:I_ideals}

In this section, we construct a family of ideals $I_h$ parametrized by Hessenberg functions $h$ that partially generalize the Tanisaki ideal.  We then establish certain fundamental properties about these ideals.  Following Biagioli, Faridi, and Rosas's construction of the Tanisaki ideal~\cite{Biag08}, we will use Young diagrams to build a set of truncated elementary symmetric functions $\mathfrak{C}_h$ that generate $I_h$.  The partial order on Hessenberg functions corresponds to the partial order of inclusion of ideals, in the sense that if $h>h'$ then $I_h \subset I_{h'}$.  (We prove this in Section~\ref{subsec:general_containment_of_I_ideals}.)  Additionally, we prove in Section~\ref{subsec:special_containment_of_I_ideals} that for certain sequences of Hessenberg functions $h>h'$ the generating sets themselves satisfy $\mathfrak{C}_h \subset \mathfrak{C}_{h'}$.  In Section~\ref{sec:min_gen_set}, we exhibit a reduced generating set for $I_h$.  In Theorem~\ref{thm:Galetto_proof}, Galetto confirms that this set is in fact minimal.

\subsection{Constructing the ideal \texorpdfstring{$I_h$}{I}}\label{subsec:I_Ideal_Construction}
We begin by defining a tableau called an {\em $h$-Ferrers diagram} that corresponds to a Hessenberg function.  This generalizes Biagioli-Faridi-Rosas's construction of the Tanisaki ideal when $\mu=(1^n)$ from $h=(1,2,\ldots,n)$ to arbitrary $h$. 

\begin{definition}[$h$-Ferrers diagram]
Let $h=(h_1,\ldots,h_n)$ be a Hessenberg function.  Draw the Ferrers diagram for a staircase partition $(1,2,\ldots,n)$ flush right and bottom. 
The {\em $h$-Ferrers diagram} is obtained by filling the bottom row with the numbers $h_1,h_2,\ldots,h_n$ from left to right, and filling with numbers in descending order up each column, as follows:
$$\setlength{\unitlength}{.24in}
\begin{picture}(24,6)(0,0)
\linethickness{.25pt}
\put(24,0){\line(-1,0){24}}
\put(24,1){\line(-1,0){24}}
\put(24,2){\line(-1,0){20}}
\put(24,4){\line(-1,0){12}}
\put(24,5){\line(-1,0){8}}
\put(24,6){\line(-1,0){4}}
\put(24,0){\line(0,1){6}}
\put(20,0){\line(0,1){6}}
\put(16,0){\line(0,1){5}}
\put(8,0){\line(0,1){3}}
\put(8,3){\line(1,0){4}}
\put(12,3){\line(0,1){1}}
\put(4,0){\line(0,1){2}}
\put(0,0){\line(0,1){1}}
\put(11.5,0.5){\ldots}
\put(11.5,1.5){\ldots}
\put(17.9,2.6){\vdots}
\put(21.9,2.6){\vdots}
\put(1.7,.25){\begin{footnotesize}\mbox{$h_1$}\end{footnotesize}}
\put(5.7,.25){\begin{footnotesize}\mbox{$h_2$}\end{footnotesize}}
\put(17.4,.25){\begin{footnotesize}\mbox{$h_{n-1}$}\end{footnotesize}}
\put(21.7,.25){\begin{footnotesize}\mbox{$h_n$}\end{footnotesize}}
\put(5.2,1.25){\begin{footnotesize}{$h_2-1$}\end{footnotesize}}
\put(16.8,1.25){\begin{footnotesize}\mbox{$h_{n-1}-1$}\end{footnotesize}}
\put(21.2,1.25){\begin{footnotesize}\mbox{$h_n-1$}\end{footnotesize}}
\put(16.2,4.25){\begin{footnotesize}\mbox{$h_{n-1}-(n-2)$}\end{footnotesize}}
\put(20.5,4.25){\begin{footnotesize}\mbox{$h_n-(n-2)$}\end{footnotesize}}
\put(20.5,5.25){\begin{footnotesize}\mbox{$h_n-(n-1)$}\end{footnotesize}}
\end{picture}$$
\end{definition}
\begin{example} The $h$-Ferrers diagram for $h=(2,3,3,5,5,6)$ is:
\medskip
$$\setlength{\unitlength}{.2in}
\begin{picture}(6,6)(0,0)
\linethickness{.25pt}
\put(6,0){\line(-1,0){6}}
\put(6,1){\line(-1,0){6}}
\put(6,2){\line(-1,0){5}}
\put(6,3){\line(-1,0){4}}
\put(6,4){\line(-1,0){3}}
\put(6,5){\line(-1,0){2}}
\put(6,6){\line(-1,0){1}}
\put(6,0){\line(0,1){6}}
\put(5,0){\line(0,1){6}}
\put(4,0){\line(0,1){5}}
\put(3,0){\line(0,1){4}}
\put(2,0){\line(0,1){3}}
\put(1,0){\line(0,1){2}}
\put(0,0){\line(0,1){1}}
\put(5.3,.25){6}\put(4.3,.25){5}\put(3.3,.25){5}
\put(2.3,.25){3}\put(1.3,.25){3}\put(.3,.25){2}
\put(5.3,1.25){5}\put(5.3,2.25){4}\put(5.3,3.25){3}\put(5.3,4.25){2}\put(5.3,5.25){1}
\put(4.3,1.25){4}\put(4.3,2.25){3}\put(4.3,3.25){2}\put(4.3,4.25){1}
\put(3.3,1.25){4}\put(3.3,2.25){3}\put(3.3,3.25){2}
\put(2.3,1.25){2}\put(2.3,2.25){1}
\put(1.3,1.25){2}
\end{picture}$$
\end{example}

We now convert an $h$-Ferrers diagram into a collection $\mathfrak{C}_h$ of truncated symmetric functions that generate the ideal $I_h$.

\begin{definition}[Generators $\mathfrak{C}_h$ and ideal $I_h$]
Let $h = (h_1, h_2, \ldots, h_n)$ be a Hessenberg function.  Define the set
$$\mathfrak{C}_h = \{e_{h_i-r}(1,\ldots,h_i)\; | \; 0 \leq r \leq i-1\}_{i=1}^n.$$
Let $I_h$ be the ideal generated by the set $\mathfrak{C}_h$, namely
\[I_h = \langle \mathfrak{C}_h \rangle.\]
\end{definition}

Note that each box in the $h$-Ferrers diagram corresponds to an element of $\mathfrak{C}_h$; the entry in the box is the degree of the corresponding truncated symmetric function, and the variable set is determined by the entry at the bottom of the box's column. 

\begin{example}\label{exam:h=(3334)_h-Ferrers}
Let $h=(3,3,3,4)$.  Then the $h$-Ferrers diagram, $\mathfrak{C}_h$, and $I_h$ are:
\medskip
$$\setlength{\unitlength}{.2in}
\begin{picture}(4,4)(0,0)
\linethickness{.25pt}
\put(4,0){\line(-1,0){4}}
\put(4,1){\line(-1,0){4}}
\put(4,2){\line(-1,0){3}}
\put(4,3){\line(-1,0){2}}
\put(4,4){\line(-1,0){1}}
\put(4,0){\line(0,1){4}}
\put(3,0){\line(0,1){4}}
\put(2,0){\line(0,1){3}}
\put(1,0){\line(0,1){2}}
\put(0,0){\line(0,1){1}}
\put(.3,.25){3}\put(1.3,.25){3}\put(2.3,.25){3}\put(3.3,.25){4}
\put(3.3,1.25){3}\put(3.3,2.25){2}\put(3.3,3.25){1}
\put(2.3,1.25){2}\put(2.3,2.25){1}
\put(1.3,1.25){2}
\end{picture}~,~$$
$$\mathfrak{C}_h = \left\{ e_1,\; e_2,\; e_3,\; e_4,\; e_1(1,2,3),\; e_2(1,2,3),\; e_3(1,2,3)\right\}, \mbox{ and}$$
$$
I_h = \left\langle e_1,\; e_2,\; e_3,\; e_4,\; x_1+x_2+x_3,\; x_1x_2+x_1x_3+x_2x_3,\;  x_1x_2x_3 \right\rangle.$$
\end{example}

We make several small observations about this construction.

\begin{remark}\label{observe:extremal_I_cases} \hspace{1in}
\begin{itemize}
\item Every collection $\mathfrak{C}_h$ contains the elementary symmetric functions $e_1,\ldots,e_n$.
\item If $h=(n,\ldots,n)$, the collection $\mathfrak{C}_h$ is the set of elementary symmetric functions in $x_1, x_2, \ldots, x_n$.
\item If $h=(1,2,\ldots,n)$, the ideal $I_h$ is the Tanisaki ideal for the partition $\mu = (1^n)$.  
\end{itemize}
\end{remark}
\begin{proof} \hspace{1in}
\begin{itemize}
\item The structure rules for Hessenberg functions require that $h_n = n$, so the far-right column of the $h$-Ferrers diagram is $1,2,\ldots,n$ for every Hessenberg function $h$.  By definition $\mathfrak{C}_h$ contains the elementary symmetric functions $e_1, e_2, \ldots, e_n$ for all $h$.
\item The bottom row of the $h$-Ferrers diagram for $h=(n,\ldots,n)$ is $n,n,n,\ldots,n$.  Thus the set $\mathfrak{C}_h$ contains exactly the $n$ distinct functions $e_i(1,\ldots,n)$ for $1\leq i\leq n$.
\item If $h=(1,2,\ldots,n)$ then the diagonal of the $h$-Ferrers diagram consists solely of ones, so the functions $e_1(1), e_1(1,2), \ldots, e_1(1,\ldots,n)$ all lie in $\mathfrak{C}_h$.  Since $x_1,\ldots,x_n$ are in the ideal, we conclude that $I_h = \left\langle x_1,x_2,\ldots,x_n \right\rangle$, which is the Tanisaki ideal for the Springer variety associated to $\mu=(1^n)$.  
\end{itemize}
\end{proof}

\begin{remark}
The Hessenberg variety for $h=(n,\ldots,n)$ is the full flag variety, and the ideal $I_h$ is the ideal generated by the elementary symmetric functions.  Borel proved that they are related: the cohomology of the full flag variety is the quotient $\mathbb{Z}[x_1,\ldots,x_n]/I_h$~\cite{Bor}.
\end{remark}

\subsection{Poset on ideals \texorpdfstring{$I_h$}{I}}\label{subsec:general_containment_of_I_ideals}

Let $h = (h_1,h_2,\ldots,h_n)$ and $h' = (h'_1,h'_2,\ldots,h'_n)$ be two Hessenberg functions.  We will show that if $h \geq h'$ (with the partial order from Definition~\ref{def:h_sequence_tree}) then $I_h$ is contained in $I_{h'}$.  In other words, there is an order-reversing morphism of posets from the poset on Hessenberg functions to the set of ideals $I_h$ partially ordered by inclusion.

\begin{definition}
We call two Hessenberg functions $h>h'$  \textit{adjacent} if there exists an edge between them in the Hasse diagram for Hessenberg functions, or equivalently if $h'_{i_0} = h_{i_0} - 1$ for some $i_0$ and if $h'_i = h_i$ for all $i\neq i_0$.
\end{definition}

We will prove the inclusion $I_h \subset I_{h'}$ by first assuming that $h > h'$ are adjacent, and then extending to arbitrary $h>h'$ using paths in the Hasse diagram for Hessenberg functions.

\begin{theorem}[Poset on ideals $I_h$]\label{thm:I_ideal_containment}
If $h > h'$ then $I_h \subset I_{h'}$.
\end{theorem}
\begin{proof}
Suppose $h = (h_1,\ldots,h_n)$ and $h'=(h_1',\ldots,h_n')$ are Hessenberg functions with $h > h'$.  There is a path from $h$ to $h'$ in the Hasse diagram on Hessenberg functions since $h>h'$.  Let $h=f_1 > f_2 > \cdots> f_s = h'$ be any such path.  If $I_{f_i} \subset I_{f_{i+1}}$ for each pair of {\em adjacent} Hessenberg functions $f_i > f_{i+1}$ in this sequence, then $I_h \subset I_{h'}$.  

Hence we prove the claim when $h>h'$ are adjacent, which we do by proving that the generators $\mathfrak{C}_h$ of $I_h$ are in $I_{h'}$.  Adjacency means that for some $i_0$ we have $h_i = h_i'$ when $i \neq i_0$ and otherwise $h_{i_0}=k=h'_{i_0}+1$, or equivalently that the $h$- and $h'$-Ferrers diagrams are identical except in column $i_0$.  It suffices to show that the $i_0$ distinct generators $e_{k-r}(1,\ldots,k) \in \mathfrak{C}_h$ for $0\leq r\leq i_0-1$ also lie in $I_{h'}$.  Lemma~\ref{lem:decomposing_truncated_elem_symfcn} says that for all $r=0,1,\ldots,i_0-1,$ we have
$$e_{k-r}(1,\ldots,k) = x_k \cdot e_{(k-r)-1}(1,\ldots,k-1) + e_{k-r}(1,\ldots,k-1).$$
Since $e_k(1,\ldots,k-1)=0$, we conclude that column $i_0$ of the $h'$-Ferrers diagram produces $e_{(k-1)-r}(1,\ldots,k-1)$ for all  $r=0,\ldots,i_0-1$.  Hence $I_{h'}$ contains $\mathfrak{C}_h$ and $I_h \subset I_{h'}$ as desired.
\end{proof}

The following result is an immediate corollary.

\begin{corollary}
If $h > h'$ and $R=\mathbb{Z}[x_1,\ldots,x_n]$, then the quotient $R/I_{h'}$ surjects naturally onto the quotient $R/I_h$.
\end{corollary}
\begin{proof}
Since $I_h$ is contained in $I_{h'}$, the claim follows.
\end{proof}

\subsection{Generator-containment sequences}\label{subsec:special_containment_of_I_ideals}
In this subsection we prove a stronger property holds: for certain Hessenberg functions $h>h'$, the generators $\mathfrak{C}_h$ of $I_h$ are contained in the set of generators $\mathfrak{C}_{h'}$ for the ideal $I_{h'}$.  This is generally false, even for adjacent Hessenberg functions.  For example if $h=(3,4,4,4)$ and $h'=(2,4,4,4)$, then $\mathfrak{C}_h = \{e_1, e_2, e_3, e_4, x_1x_2x_3\}$ is not a subset of $\mathfrak{C}_{h'} = \{e_1, e_2, e_3, e_4, x_1x_2\}$, though $I_h \subset I_{h'}$ by Theorem~\ref{thm:I_ideal_containment}.  

First we define {\em generator-containment sequences} of sets $\mathfrak{C}_h$.  In Lemma~\ref{lem:sufficient_gen_containment_condition}, we give a sufficient condition to ensure that $\mathfrak{C}_h \subset \mathfrak{C}_{h'}$ form a generator-containment sequence.  Theorem~\ref{thm:generator_containment_sequences_connect_graph} proves that the generator-containment sequences of Hessenberg functions induce a spanning subgraph in the Hasse diagram of Hessenberg functions.

\begin{definition}
Let $h=f_1 > f_2 > \cdots > f_r=h'$ be a sequence of Hessenberg functions such that $f_{i}$ and $f_{i+1}$ are adjacent for each $i \leq r-1$.  The sequence is a \textit{generator-containment sequence} if $\mathfrak{C}_{f_{i}} \subset \mathfrak{C}_{f_{i+1}}$ for each $i \leq r-1$.
\end{definition}

\begin{lemma}\label{lem:sufficient_gen_containment_condition}
Suppose $h>h'$ are adjacent Hessenberg functions and that $i_0$ is the index with $h_{i_0} = h_{i_0}'+1$.  If $h_{i_0} = h'_k$ for some $k > i_0$ then $\mathfrak{C}_h \subset \mathfrak{C}_{h'}$.
\end{lemma}

\begin{proof}
By definition of adjacency, we know that $h_{i_0} = h_{i_0}'+1$ and $h_i = h_i'$ for all $i \neq i_0$, or equivalently that the $h$-Ferrers and $h'$-Ferrers diagrams differ only in the $i_0^{th}$ column.  It suffices to show that the generators corresponding to the entries of column $i_0$ in the $h$-Ferrers diagram also lie in $\mathfrak{C}_{h'}$, namely that the functions $e_{h_{i_0}-r}(1,\ldots,h_{i_0}) \in \mathfrak{C}_h$ for $0 \leq r \leq i_0-1$ are also in $\mathfrak{C}_{h'}$. Suppose that $h'_k=h_{i_0}$ for some $k > i_0$.  Then some column to the right of column $i_0$ in the $h'$-Ferrers diagram has the value $h_{i_0}$ in its bottom box.  This column is taller than column $i_0$ in the $h$-Ferrers diagram.  Thus the generators $e_{h_{i_0}-r}(1,\ldots,h_{i_0})$ for $0 \leq r \leq i_0-1$ lie in $\mathfrak{C}_{h'}$.
\end{proof}

\begin{theorem}\label{thm:generator_containment_sequences_connect_graph}
For each Hessenberg function $h>(1,2,\dots,n)$ there exists at least one adjacent function $h'$ with both $h>h'$ and $\mathfrak{C}_h \subset \mathfrak{C}_{h'}$.
\end{theorem}

\begin{proof}
Fix $h=(h_1,\ldots,h_n)$ with $h>(1,2,\dots,n)$.  Find the maximal $i_0$ for which both:
\begin{align*}
(1) &\hspace{.25in} i_0 \leq h_{i_0}-1\;\; \mbox{and}\\
(2) &\hspace{.25in} h_{i_0-1} \neq h_{i_0}.
\end{align*}
At least one $i_0$ satisfies Condition (1) since $h>(1,2,\dots,n)$.  If $h_{i_0-1}=h_{i_0}$ then $h_{i_0-1}$ also satisfies Condition (1).  So at least one $i_0$ satisfying both conditions exists.  Define the function $h'$ by $h'_{i_0}=h_{i_0}-1$ and $h'_i=h_i$ for all $i\neq i_0$.  Note that $h'$ is a Hessenberg function since $h'_{i_0} \geq h'_{i_0-1}$ by Condition (2).  

It suffices to show that there exists a value $h'_k$ with $k>i_0$ so that $h'_k=h_{i_0}$.  Then by Lemma~\ref{lem:sufficient_gen_containment_condition} we can conclude $\mathfrak{C}_h \subset \mathfrak{C}_{h'}$.  

We claim $k=i_0+1$ works.  If not then $h'_{i_0+1}>h_{i_0}$ because Hessenberg functions are nondecreasing.  Hence $$h_{i_0+1} = h'_{i_0+1} > h_{i_0} \geq i_0+1,$$ where the last inequality arises from Condition $(1)$.  Thus $h_{i_0+1}$ satisfies Conditions $(1)$ and $(2)$, contradicting the maximality of $i_0$.  We conclude $h'_{i_0+1} = h_{i_0}$ as desired.
\end{proof}

The following corollary highlights the main conclusions of Theorem~\ref{thm:generator_containment_sequences_connect_graph}.

\begin{corollary}\label{cor:All_Paths_Lead_To_Rome}
There is a generator-containment sequence from each Hessenberg function $h$ to the minimal Hessenberg function $(1,2,\ldots,n)$.  In particular,
\begin{itemize}
\item  all generating sets $\mathfrak{C}_h$ are contained in the set $\mathfrak{C}_{(1,2,\ldots,n)}$, and
\item these generator-containment sequences form a spanning subgraph in the Hasse diagram of Hessenberg functions.
\end{itemize}
\end{corollary}

Figure~\ref{fig:h_beta_trees}.(a) and Figure~\ref{fig:Hess_Fcn_Hasse_n=5} show the Hasse diagrams on Hessenberg functions for $n=4$ and $n=5$, respectively, with generator-containment sequences indicated using double-lined dashed edges.

\subsection{Reduced generating set for \texorpdfstring{$I_h$}{I}}\label{sec:min_gen_set}

An $h$-Ferrers diagram determines $\frac{n(n+1)}{2}$ generators for the ideal $I_h$.  The generating set $\mathfrak{C}_h$ of the generalized Tanisaki ideal $I_h$ is often highly nonminimal, like the generators of the original Tanisaki ideal.  In this section, we construct a reduced generating set for $I_h$ with only $n$ generators, using the functions corresponding to the antidiagonal of the $h$-Ferrers diagram.  We conjectured---and Galetto proved in the Appendix to this manuscript---that this set of antidiagonal generators gives a minimal generating set for $I_h$.  We further conjecture that similar methods could provide a minimal generating set for the Tanisaki ideal (see Section~\ref{section:questions}).

\begin{definition}[Antidiagonal ideal]\label{def:anti_diagonal_I_ideal}
Let $h=(h_1,\ldots,h_n)$ be a Hessenberg function.  The \textit{antidiagonal ideal} $\IAD$ in $I_h$ is the ideal generated by the functions corresponding to the boxes on the antidiagonal of the $h$-Ferrers diagram.  That is,
$$\IAD = \left\langle e_{h_i-(i-1)}(1,\ldots,h_i) \right\rangle_{i=1}^n \subseteq I_h.$$
\end{definition}

In fact we can show that each generator in $\mathfrak{C}_h$ lies in the antidiagonal ideal.  The next lemma proves this in a special case.

\begin{lemma}\label{lem:bottom_row_is_fine}
Let $h=(h_1,\ldots,h_n)$ be a Hessenberg function.  For each $i \in \{1,2,\ldots,n\}$, the generator $e_{h_i}(1,\ldots,h_i)$ lies in $\IAD$.
\end{lemma}
\begin{proof}
Consider $e_{h_i}(1,\ldots,h_i) \in \mathfrak{C}_h$.  When $i=1$, the claim is vacuously true.  If $i>1$ then
\[e_{h_i}(1,\ldots,h_i) 	= x_1 \cdots x_{h_i}= (x_{h_1+1} \cdots x_{h_i}) \cdot e_{h_1}(1,\ldots,h_1).\]
Hence $e_{h_i}(1,\ldots,h_i)$ is a multiple of $e_{h_1}(1,\ldots,h_1)$, so $e_{h_i}(1,\ldots,h_i) \in \IAD$.
\end{proof}

The previous lemma is the base case for an inductive proof in the next theorem.

\begin{theorem}[Reduced generating set for $I_h$]\label{thm:min_gen_set_for_I} Let $h=(h_1,\ldots,h_n)$ be a Hessenberg function.  Then $I_h \subseteq \IAD$ and in particular $I_h$ is generated by the generators of $\IAD$.
\end{theorem}

\begin{proof}
We will show that the generators of $I_h$ that correspond to boxes off the antidiagonal of the $h$-Ferrers diagram lie in $\IAD$, namely that if $2 \leq i \leq n$ and $0 \leq r \leq i-2$ then $e_{h_i-r}(1,\ldots,h_i) \in \IAD$.  (The only box in the first column is on the antidiagonal, so $e_{h_1}(1,\ldots,h_1) \in \IAD$ by definition.)

We induct on the columns of the $h$-Ferrers diagram moving left to right.  The base case is $i=2$.  It holds since $e_{h_2}(1,\ldots,h_2)$ lies in $\IAD$ by Lemma~\ref{lem:bottom_row_is_fine}, and the antidiagonal generator $e_{h_2-1}(1,\ldots,h_2)$ is by definition in $\IAD$.  

Assume for some column $i$ that the function $e_{h_i-r}(1,\ldots,h_i)$ is in the ideal $\IAD$ for all $r \in \{0,\ldots,i-2\}$.  We now show that for column $i+1$ and  for all $r \in \{0,\ldots,i-1\}$, the generator $e_{h_{i+1}-r}(1,\ldots,h_{i+1}) \in \IAD$.   Consider the following schematic of columns $i$ and $i+1$:
\smallskip
$$\setlength{\unitlength}{.225in}
\begin{picture}(8,9)(0,0)
\linethickness{.25pt}
\put(0,0){\line(0,1){8}}
\put(4,0){\line(0,1){9}}
\put(8,0){\line(0,1){9}}
\put(0,0){\line(1,0){8}} \put(0,1){\line(1,0){8}} \put(0,2){\line(1,0){8}}
\put(0,7){\line(1,0){8}} \put(0,8){\line(1,0){8}}
\put(4,9){\line(1,0){4}}
\put(1.8,.35){\begin{footnotesize}\mbox{$h_i$}\end{footnotesize}}
\put(5.4,.35){\begin{footnotesize}\mbox{$h_{i+1}$}\end{footnotesize}}
\put(1.3,1.35){\begin{footnotesize}\mbox{$h_i-1$}\end{footnotesize}}
\put(4.9,1.35){\begin{footnotesize}\mbox{$h_{i+1}-1$}\end{footnotesize}}
\put(2,2.5){\vdots} \put(6,3.1){\vdots}
\put(2,5.9){\vdots} \put(6,5.9){\vdots}
\put(.4,7.35){\begin{footnotesize}\mbox{$h_i-(i-1)$}\end{footnotesize}}
\put(4.2,7.35){\begin{footnotesize}\mbox{$h_{i+1}-(i-1)$}\end{footnotesize}}
\put(4.9,8.35){\begin{footnotesize}\mbox{$h_{i+1}-i$}\end{footnotesize}}

\put(0,3.5){\line(1,0){4}} \put(0,4.5){\line(1,0){8}} \put(0,5.5){\line(1,0){8}}
\put(1.3,4.85){\begin{footnotesize}\mbox{$h_i-s$}\end{footnotesize}}
\put(4.9,4.85){\begin{footnotesize}\mbox{$h_{i+1}-s$}\end{footnotesize}}
\put(.4,3.85){\begin{footnotesize}\mbox{$h_i-(s-1)$}\end{footnotesize}}
\end{picture}$$
If $h_i = h_{i+1}$ then the inductive step applies, since 
\[e_{h_{i+1}-r}(1,\ldots,h_{i+1}) =e_{h_{i}-r}(1,\ldots,h_{i})\] 
for all $r \in \{0,\ldots,i-1\}$.  Assume instead that $h_i < h_{i+1}$ and consider $h_{i+1}-s$.  If $s=0$ then $e_{h_{i+1}-s}(1,\ldots,h_{i+1})$ lies in $\IAD$ by Lemma~\ref{lem:bottom_row_is_fine}.  Suppose that $s\in\{1,\ldots,i-1\}$.  Choose $r$ so that $h_{i+1}-r = h_i$.  By Lemma~\ref{lem:cool_combinatorial_sym_func_lemma} we may write 
\begin{align*}
e_{h_{i+1}-s}(1,\ldots,h_{i+1}) = \sum_{t=0}^{h_{i+1}-h_i} e_t(h_i+1,\ldots,h_{i+1}) \cdot e_{h_{i+1}-s-t}(1,\ldots,h_i).
\end{align*}

\vspace{.1in}

\noindent We need to verify that the degrees $h_{i+1}-s-t$ are values in column $i$ whenever the function $e_{h_{i+1}-s-t}(1,\ldots,h_i)$ is nonzero.  Since $t \leq h_{i+1}-h_i$ we have
$$h_{i+1}-s-(h_{i+1}-h_i) \leq h_{i+1}-s-t $$
and so $h_i-s \leq h_{i+1}-s-t$.  Recall from Remark~\ref{rem:nonzero_terms_in_cool_sym_func_lemma} that $e_{h_{i+1}-s-t}(1,\ldots,h_i)$ is zero unless $h_{i+1}-s-t \leq h_i$.  Hence we assume $h_{i+1}-s-h_i \leq t$ and so $h_{i+1}-s-t \leq h_i$.  

We conclude that the degrees $h_{i+1}-s-t$ satisfy 
\[h_i - s \leq h_{i+1}-s-t \leq h_i,\] 
namely they are values in column $i$.  By the induction hypothesis, each $e_{h_{i+1}-s-t}(1,\ldots,h_i)$ lies in $\IAD$ and hence $e_{h_{i+1}-s-t}(1,\ldots,h_{i+1})$ also lies in $\IAD$.  Thus $I_h \subseteq \IAD$.
\end{proof}




\begin{theorem}[Galetto~\cite{Galetto}]\label{thm:Galetto_proof}
The generators of $\IAD$ form a minimal generating set for $I_h$.
\end{theorem}

Galetto's observation uses tools from commutative algebra together with results from later sections of this paper; the proof can be found in Appendix~\ref{app:minimal_gen_proof}.


\section{Family of Hessenberg basis ideals \texorpdfstring{$J_h$}{J}}\label{sec:J_ideals}

In the last section we generalized the Tanisaki ideal to a family of ideals $I_h$.  In this section we develop machinery to give a Gr\"{o}bner basis for each ideal $I_h$.  We do this by constructing a family of ideals $J_h$ called \textit{Hessenberg basis ideals}.  We build the ideals $J_h$ using truncated complete symmetric functions, similarly to how we built the ideals $I_h$ from truncated elementary symmetric functions.  The ideals $J_h$ have several useful properties, including:
\begin{enumerate}
\item For each Hessenberg function $h=(h_1,\ldots,h_n)$, we define a set $\mathcal{J}_h$ of $n$ polynomials that generate the corresponding ideal $J_h$. (Definition~\ref{def:J-ideal_Construction})
\item The generators of $J_h$ form a Gr\"obner basis. (Theorem~\ref{thm:J_ideal_Grob})
\item Let $R=\mathbb{Z}[x_1,\ldots,x_n]$.  Then $R/J_h$ has finite rank, and a monomial basis for $R/J_h$ is easily obtained from the degree tuple corresponding to $h$. (Theorem~\ref{thm:Basis_for_R/J} and Corollary~\ref{thm:dim(R/J)})
\end{enumerate}
The first author proved that the monomial basis for $R/J_h$ gives the Betti numbers for regular nilpotent Hessenberg varieties~\cite{Mb1}, as discussed in more detail in Section~\ref{subsec:cohom_applications}.  The ideals $J_h'$ obtained from $J_h$ under the change of variables $x_i \leftrightarrow x_{n-i+1}$ also appear in work of Ding, Gasharov-Reiner, and Develin-Martin-Reiner, where the quotient $R/J_h'$ is proven to be the cohomology ring of a class of smooth Schubert varieties~\cite{GasRei}, \cite{DevMarRei} and where the Gr\"{o}bner basis for $J_h'$ was first noted~\cite[Remark 3.3]{GasRei}.

\subsection{First properties of the ideal \texorpdfstring{$J_h$}{J}} \label{subsec:constructing_the_ideal_J}

The generators of the ideal $J_h$ are naturally parametrized by the degree tuple $\beta$ corresponding to $h$ rather than the Hessenberg function itself.  Given a Hessenberg function $h=(h_1,\ldots,h_n)$, recall from Lemma~\ref{lemma:simple_formula_to_compute_beta} that the corresponding degree tuple $\beta=(\beta_n,\beta_{n-1},\ldots,\beta_1)$ is defined by 
\[\beta_i = i - \#\{h_k | h_k < i\}  \textup{     for each     } i \in \{1,\ldots,n\}.\] 

\begin{definition}[The ideal $J_h$]\label{def:J-ideal_Construction}
Let $h=(h_1,\ldots,h_n)$ be a Hessenberg function with corresponding degree tuple $\beta=(\beta_n,\beta_{n-1},\ldots,\beta_1)$.  We define a set of polynomials $\mathcal{J}_h$ by 
\[\mathcal{J}_h : = \{\tilde{e}_{\beta_n}(n),\tilde{e}_{\beta_{n-1}}(n-1,n),\ldots,\tilde{e}_{\beta_1}(1,\ldots,n)\}\]
and we define the \textit{Hessenberg basis ideal} $J_h$ by
$$J_h:=\langle \mathcal{J}_h \rangle = \langle \tilde{e}_{\beta_n}(n),\tilde{e}_{\beta_{n-1}}(n-1,n),\ldots,\tilde{e}_{\beta_1}(1,\ldots,n)\rangle.$$
\end{definition}

\begin{example}\label{example: J_h}
The Hessenberg function $h=(3,3,3,4)$ corresponds to the degree tuple $\beta=(1,3,2,1)$ so $J_h = \langle \tilde{e}_1(4),\tilde{e}_3(3,4),\tilde{e}_2(2,3,4),\tilde{e}_1(1,2,3,4) \rangle$.  That is,
$$J_h = \left(
\begin{array}{c}
x_4, \\
x_3^3 + x_3^2 x_4 + x_3 x_4^2 + x_4^3, \\
x_2^2 + x_2 x_3 + x_2 x_4 + x_3^2 + x_3 x_4 + x_4^2, \\
x_1 + x_2 + x_3 + x_4
\end{array}
\right).$$
\end{example}
\medskip

The ideals $J_h$ are partially ordered by inclusion.  As with the ideals $I_h$, this poset has the same structure as the poset on Hessenberg functions, in the following sense.

\begin{theorem}[Poset on ideals $J_h$]\label{thm:J_ideal_containment}
Let $\beta'$ be the degree tuple for $h'$.  
\begin{enumerate}
\item Choose $i \in \{1,2,\ldots,n\}$. If $d > \beta'_i$ then $\tilde{e}_d(i,i+1,\ldots,n) \in J_{h'}$.
\item If $\beta$ is the degree tuple corresponding to $h$ and $\beta > \beta'$ then $J_h \subseteq J_{h'}$.
\end{enumerate}
\end{theorem}

\begin{proof}
Our proof is by induction.  The inductive hypothesis is that whenever $d \geq \beta'_{i+1}$ the function $\tilde{e}_d(i+1,i+2,\ldots,n) \in J_{h'}$.  The base case occurs when $i+1=n$.  In this case we have $\tilde{e}_d(n) = x^d_n$ and so $\tilde{e}_d(n) \in J_{h'}$ whenever $d \geq \beta'_n$.

To prove the inductive step, we partition the terms of $\tilde{e}_d(i,i+1,\ldots,n)$ according to the power of $x_i$.  We will show that
\begin{equation} \label{eqn:J-containment}
\tilde{e}_d(i,i+1,\ldots,n) = x_i^{d-\beta'_i} \cdot \tilde{e}_{\beta'_i}(i,i+1,\ldots,n) + \sum_{t=0}^{d-\beta'_i-1} x_i^t \cdot \tilde{e}_{d-t}(i+1,i+2,\ldots,n).
\end{equation}
Suppose that $x_i^{\alpha_i}$ is the power of $x_i$ that appears in a given monomial term of the function $\tilde{e}_d(i,i+1,\ldots,n)$.  Consider the following cases separately: when $\alpha_i \geq d-\beta'_i$ and when $\alpha_i = t$ for each $t=0,1,2,\ldots,d-\beta'_i-1$.  On the one hand, the terms in which $\alpha_i \geq d-\beta'_i$ are exactly the terms of $x_i^{d-\beta'_i} \cdot \tilde{e}_{\beta'_i}(i,i+1,\ldots,n)$.  On the other hand, the terms with $\alpha_i=t$ are exactly the terms $x_i^t \cdot \tilde{e}_{d-t}(i+1,i+2,\ldots,n)$, for each $t$ with $0 \leq t \leq d-\beta'_i-1$.  This proves Equation~\eqref{eqn:J-containment}.

The sum in Equation~\eqref{eqn:J-containment} contains those $\tilde{e}_{d-t}(i+1,i+2,\ldots,n)$ for which
\[\beta'_i+1 \leq d-t \leq d.\]
Definition~\ref{def:deg_tuples}.(b) of degree tuples guarantees that $\beta'_{i+1} \leq \beta'_i+1$, and hence the induction hypothesis ensures that each $\tilde{e}_{d-t}(i+1,i+2,\ldots,n)$ lies in $J_{h'}$.  We conclude that if $d \geq \beta'_{i}$ then $\tilde{e}_d(i,i+1,\ldots,n) \in J_{h'}$.  

By induction, for each $i$ we have $\tilde{e}_d(i,i+1,\ldots,n) \in J_{h'}$ whenever $d \geq \beta'_{i}$.  This proves Part~(1).  In particular if $\beta > \beta'$ then $\tilde{e}_{\beta_i}(i,i+1,\ldots,n) \in J_{h'}$ for each $i$.  We conclude that if $\beta > \beta'$ then $J_{h} \subset J_{h'}$ as desired.
\end{proof}

\subsection{Generators form a Gr\"obner basis}\label{subsec:J_is_Grobner}
We now prove that the set $\mathcal{J}_h$ forms a Gr\"{o}bner basis for $J_h$.  More detail on Gr\"obner bases can be found in classical texts such as~\cite{Cox}.

Let $R$ be the polynomial ring $\mathbb{Z}[x_1,\ldots,x_n]$.  Let $\X$ and $\Y$ be the monomials $x_1^{\alpha_1}x_2^{\alpha_2}\cdots x_n^{\alpha_n}$ and $x_1^{\beta_1}x_2^{\beta_2}\cdots x_n^{\beta_n}$ with exponents $\alpha=(\alpha_1,\ldots,\alpha_n)$ and $\beta=(\beta_1,\ldots,\beta_n)$ respectively.  Let $lex$ denote the \textit{lexicographic monomial ordering} in $R$.  In other words, if the leftmost nonzero entry in the vector $\alpha - \beta \in \mathbb{Z}^n$ is positive, then $\X >_{lex} \Y$.  For example $x_1 x_2^2 >_{lex} x_1 x_2$ since $\alpha - \beta = (0,1)$.

\begin{remark}\label{rem:term_ordering_does_not_matter}
\textit{Graded lexicographic order}, or $grlex$, is another common monomial order.  We deal exclusively with homogeneous functions, where $lex$ and $grlex$ coincide.
\end{remark}

The \textit{leading monomial} of a polynomial $f$ in $R$, denoted $LM(f)$, is the term whose monomial is greatest with respect to the ordering.  For example, the leading monomial of $x_3^4 + x_1x_2^3 + x_1^2x_2x_4 + x_1x_2x_3x_4$ is the monomial $x_1^2x_2x_3$.  The \textit{leading term} of a polynomial $f$, denoted $LT(f)$, is $LM(f)$ scaled by its coefficient (if any).   In our applications, $LM(f)$ and $LT(f)$ coincide, since each monomial in our symmetric functions has coefficient one.

If $I$ is an ideal in $R$ then $\LT{I}$ denotes the ideal generated by the leading terms of each element in $I$.  If $I$ is finitely generated, say by $f_1, \ldots, f_s$, then the ideal $\left\langle LT(f_1), \ldots, LT(f_s) \right\rangle$ is contained in $\LT{I}$.  The $f_i$ must form a {\em Gr\"{o}bner basis} for the converse to be true.

\begin{definition}[Gr\"obner basis]\label{def:Grobner_Basis}
The set $G=\{g_1,\ldots,g_t\}$ is a Gr\"obner basis for an ideal $I$ in $R$ if and only if $\LT{I} = \left\langle LT(g_1),\ldots,LT(g_t) \right\rangle$.
\end{definition}

Denote the least common multiple by LCM.  Polynomials $f_1 \neq f_2$ are relatively prime if
$$LCM(LM(f_1),LM(f_2)) = LM(f_1)\cdot LM(f_2).$$
Proposition~\ref{prop:Sufficient_Criterion_for_Grobner} gives one of many sufficient criteria to determine whether a set of polynomials $G=\{g_1,\ldots,g_t\}$ forms a Gr\"obner basis for the ideal they generate.

\begin{proposition}[Cox-Little-O'Shea~\cite{Cox}]\label{prop:Sufficient_Criterion_for_Grobner}
Let $G=\{g_1,\ldots,g_t\}$ be a set of polynomials in $R$.  If the leading monomials of the polynomials in $G$ are pairwise relatively prime, then $G$ is a Gr\"obner basis for the ideal they generate.
\end{proposition}

We use this to prove that $\mathcal{J}_h$ forms a Gr\"{o}bner basis for $J_h$, extending a classical result of Cauchy-Valibouze-Mora-Sala for $h=(n,n,\ldots,n)$, as described in Theorem~\ref{thm:Maximal_case_equality_of_I_and_J}.  After this manuscript went to press, we learned that Gasharov-Reiner proved a version of this theorem up to the change of variables $x_i \leftrightarrow x_{n-i+1}$ \cite[Remark 3.3]{GasRei}.

\begin{theorem}\label{thm:J_ideal_Grob}
The generating set $\mathcal{J}_h$ is a Gr\"obner basis for the ideal $J_h$ with respect to $lex$ or $grlex$ orderings.
\end{theorem}

\begin{proof}
Denote the generators by $f_i = \tilde{e}_{\beta_i}(i,\ldots,n)$ for $1 \leq i \leq n$.  To show that the generating set $\mathcal{J}_h=\{f_1,f_2,\ldots,f_n\}$ is a Gr\"obner basis, it suffices to show that the leading monomials of $f_i$ and $f_j$ are relatively prime for all $i \neq j$, by Proposition~\ref{prop:Sufficient_Criterion_for_Grobner}.  The leading monomial of $f_i$ is $x_i^{\beta_i}$ for each $i$ for both lex and grlex.  If $i\neq j$ then 
$$LCM(LM(f_i),LM(f_j)) = LCM(x_i^{\beta_i}, x_j^{\beta_j}) = x_i^{\beta_i}x_j^{\beta_j} = LM(f_i)\cdot LM(f_j).$$
We conclude that $\mathcal{J}_h$ is a Gr\"obner basis for the ideal $J_h$, as desired.
\end{proof}

\begin{remark}\label{rmk:frank's remark}
The analogous claim for $h=(n,n,\ldots,n)$ is a classical result cited in Theorem~\ref{thm:Maximal_case_equality_of_I_and_J}.  In that result, the functions $\mathcal{J}_h$ form a Gr\"{o}bner basis for {\em all} monomial term orders of the form $x_{\pi(1)}<\cdots<x_{\pi(n)}$ for $\pi \in S_n$.  This is false in our generality because the functions $\tilde{e}_{\beta_i}(i,\ldots,n)$ are rarely symmetric.  For instance, if $n=4$ and the term order is $x_4 > x_3 > x_2 > x_1$, then each generator in $\mathcal{J}_h$ has leading term $x_4^d$ for some $d$.  However, the ideal in Example~\ref{example: J_h} contains $x_1+x_2+x_3$, whose leading term is $x_3 \not \in  \langle x_4\rangle$.
\end{remark}

\subsection{The quotient ring \texorpdfstring{$R/J_h$}{R/J} and regular nilpotent Hessenberg varieties}\label{subsec:cohom_applications}
Gr\"{o}bner bases for the ideal $I$ can be used to construct a simple, elegant basis for the quotient $R/I$.  The quotient $R/J_h$ gives the Betti numbers for a family of varieties called regular nilpotent Hessenberg varieties, as proven in earlier work by the first author~\cite[Theorem 3.3.3]{Mb1}.  In this subsection we construct the basis for $R/J_h$ used in his work.

Cox-Little-O'Shea sketch a proof of the following~\cite{Cox}; details are in~\cite[Appendix A.2]{Mb-thesis}.

\begin{proposition}[Cox-Little-O'Shea~\cite{Cox}]\label{thm:R/I_Basis}
Let $I$ be an ideal in $R$.   The $\mathbb{Q}$-span of the quotient $R/I$ is isomorphic to the $\mathbb{Q}$-span of the set $\{\X \;|\; \X \notin \LT{I} \}$ as $\mathbb{Q}$-vector spaces.
\end{proposition}

To find a basis for $R/J_h$ we must understand more precisely the ideal $\LT{J_h}$ generated by the leading terms of elements in the ideal $J_h$.

\begin{lemma}
Fix  a Hessenberg function $h=(h_1,\ldots,h_n)$.  The ideal $\LT{J_h}$ is the monomial ideal $\langle x_1^{\beta_1},x_2^{\beta_2},\ldots,x_n^{\beta_n} \rangle$.
\end{lemma}

\begin{proof}
Denote $f_i = \tilde{e}_{\beta_i}(i,\ldots,n)$ for each $1 \leq i \leq n$.  The set $\{f_1, f_2, \ldots, f_n\}$ is a Gr\"{o}bner basis for the ideal $J_h$ by Theorem~\ref{thm:J_ideal_Grob}.  By definition this means 
\[\LT{J_h} = \left\langle LT(f_1),\ldots,LT(f_n) \right\rangle.\]  
Since each $LT(f_i)=x_i^{\beta_i}$, the ideal $\left\langle LT(f_1),\ldots,LT(f_n) \right\rangle = \langle x_1^{\beta_1},x_2^{\beta_2},\ldots,x_n^{\beta_n} \rangle$ as desired.
\end{proof}

The next two claims follow quickly from the results we have assembled.

\begin{theorem}[Basis for $R/J_h$]\label{thm:Basis_for_R/J}  Let $J_h$ be the ideal corresponding to the Hessenberg function $h=(h_1,\ldots,h_n)$.  Then $R/J_h$ has basis
$$\left\lbrace x_1^{\alpha_1} x_2^{\alpha_2} \cdots x_n^{\alpha_n} \;\; \vline \;\; 0\leq\alpha_i\leq\beta_i-1,\;\; i=1,\ldots,n \right\rbrace.$$
\end{theorem}

\begin{proof}
By Theorem~\ref{thm:R/I_Basis}, the quotient $R/J_h$ has basis $\{\X \;|\; \X \notin \LT{J_h} \}$.  By definition $\X \notin \LT{J_h}$ implies none of the $x_i^{\beta_i}$ divides $\X$.  Thus the exponent of $x_i$ in the monomial $\X$ cannot exceed $\beta_i-1$.  So $\X=x_1^{\alpha_1} x_2^{\alpha_2} \cdots x_n^{\alpha_n}$ must satisfy $\alpha_i \in \{0,1,\ldots,\beta_i-1\}$ for each $1 \leq i \leq n$, as desired.
\end{proof}

\begin{corollary}\label{thm:dim(R/J)}
For every Hessenberg function $h$,  the rank of $R/J_h$ equals $\displaystyle \prod_{i=1}^n \beta_i$.  In particular, the quotient $R/J_h$ has finite rank.
\end{corollary}




\section{Equality of the two families of ideals \texorpdfstring{$I_h=J_h$}{I=J}} \label{sec:Ideal_equality_section}

This section contains the main results of this paper.  Having defined $I_h$ and $J_h$ earlier, we will now prove that $I_h = J_h$.  Proposition~\ref{prop:crucial_identity} establishes a remarkable relationship between truncated elementary symmetric functions of degree $d$ in variables $x_1,\ldots,x_r$ and truncated complete symmetric functions of the same degree in variables $x_{r+1},\ldots,x_n$.  Corollary~\ref{cor:elem_and_complete_symfcn_identity} interprets Proposition~\ref{prop:crucial_identity} in terms of the ideals $I_h$ and $J_h$.  We break the proof of the main result, Theorem~\ref{thm:I_h=J_h}, into two pieces: first, that the antidiagonal ideal 
$\IAD \subseteq J_h$; and second, that $J_h \subseteq I_h$.  We then use Corollary~\ref{cor:elem_and_complete_symfcn_identity} to prove both pieces by induction.  Since $\IAD = I_h$ by Theorem~\ref{thm:min_gen_set_for_I}, this completes the proof of Theorem~\ref{thm:I_h=J_h}.

We begin with a classical identity that expresses elementary symmetric functions in terms of truncated complete symmetric functions.

\subsection{Gr\"obner basis for the set of elementary symmetric functions}\label{subsec:Grob_basis_for_elem_sym}

In this subsection we confirm that the ideals $I_h$ and $J_h$ coincide when $h$ is maximal, namely when $h=(n,\ldots,n)$, when $I_h$ is generated by the elementary symmetric functions, and when $J_h = \left\langle \tilde{e}_i(i,\ldots,n) \right\rangle_{i=1}^n$.  

In this case the equality $I_h = J_h$ together with the Gr\"{o}bner basis for $J_h$ gives a Gr\"obner basis for the set of elementary symmetric functions.  This is a much-studied problem.  Theorem~\ref{thm:Maximal_case_equality_of_I_and_J} is an identity proven in 2003 by Mora and Sala~\cite{MS} (up to change of variables).  Mora told us that their result reproves a 1994 result in Valibouze's thesis~\cite{AV}.  Valibouze in turn told us that her work reproves an 1840 result by Cauchy~\cite{AC} in the case $n=4$.

\begin{theorem}[Cauchy-Valibouze-Mora-Sala] \label{thm:Maximal_case_equality_of_I_and_J}
For $1\leq d\leq n$, the elementary symmetric function $e_d$ has the following presentation:
$$e_d = \sum_{t=1}^d (-1)^{t+1} \cdot e_{d-t} (t+1,\ldots,n) \cdot \tilde{e}_t(t,\ldots,n).$$
We conclude that if $h=(n,\ldots,n)$, then $I_h \subseteq J_h$.
\end{theorem}

\begin{proof}
Mora-Sala prove the following~\cite[Proposition 2.1]{MS}:
$$e_d + \sum_{t=1}^{d-1} (-1)^t \tilde{e}_t(1, \ldots,n-t+1) e_{d-t}(1,\ldots,n-t) + (-1)^d \tilde{e}_d(1, \ldots,n-d+1) = 0.$$
Also, they prove the identity holds for any term order of the form $x_{\pi(1)}<\cdots<x_{\pi(n)}$ for $\pi \in S_n$~\cite[Proposition 2.2]{MS}.  When the term order is defined by $\pi(i) = n-i+1$ for each $i=1,2,\ldots,n$, we obtain
$$e_d = \sum_{t=1}^d (-1)^{t+1} \cdot e_{d-t} (t+1,\ldots,n) \cdot \tilde{e}_t(t,\ldots,n)$$
for $1\leq d\leq n$.  When $h=(n,n,\ldots,n)$ we have the ideal $I_h=\left\langle e_1,\ldots,e_n\right\rangle$ and the ideal $J_h=\left\langle \tilde{e}_i(i,\ldots,n) \right\rangle_{i=1}^n$.  We conclude $I_h \subseteq J_h$ as desired.
\end{proof}

The next corollary uses the explicit identities in the previous theorem to conclude that in fact $I_h=J_h$ when $h=(n,\ldots,n)$.

\begin{corollary}\label{cor:I_equals_J}
If $h=(n,\ldots,n)$ then $J_h\subseteq I_h$ and so $I_h = J_h$.
\end{corollary}

\begin{proof}
Consider the matrix $B=(b_{ij})$ whose entries consist of the coefficients in Theorem~\ref{thm:Maximal_case_equality_of_I_and_J}, namely
\begin{equation*}
b_{ij} =
\begin{cases}
(-1)^{j+1} \cdot e_{i-j} (j+1,\ldots,n) & \text{if $j \leq i$}, \\
0 & \text{if $j>i$}.
\end{cases}
\end{equation*}
Theorem~\ref{thm:Maximal_case_equality_of_I_and_J} states that 
\[B \left(\tilde{e}_1(1,\ldots,n), \tilde{e}_2(2,\ldots,n), \ldots, \tilde{e}_n(n)\right)^T = \left(e_1, e_2, \ldots, e_n \right)^T.\] 
By construction $B$ is a lower-triangular matrix with ones along the diagonal.  This means $B$ is invertible~\cite[Chapter I.6]{Mac95}.  The $i^{th}$ row of the matrix equation
\[B^{-1} \left(e_1, e_2, \ldots, e_n \right)^T = \left(\tilde{e}_1(1,\ldots,n), \tilde{e}_2(2,\ldots,n), \ldots, \tilde{e}_n(n)\right)^T\] 
expresses $\tilde{e}_i(i,\ldots,n)$ in terms of the elementary symmetric functions.  For the Hessenberg function $h=(n,n,\ldots,n)$ we have $I_h=\left\langle e_1,\ldots,e_n\right\rangle$ and $J_h=\left\langle \tilde{e}_i(i,\ldots,n) \right\rangle_{i=1}^n$. Hence $J_h \subseteq I_h$ as desired.  With Theorem~\ref{thm:Maximal_case_equality_of_I_and_J}, we conclude $J_h=I_h$.
\end{proof}

\begin{example}\label{exam:B_matrix_and_inverse}
When $n=4$, the matrices $B$ and $B^{-1}$ are:
\begin{figure}[h]
\begin{center}
\subfigure[The matrix $B$]{
$\left(\begin{array}{cccc}
1 & 0 & 0 & 0 \\
e_1(x_{234}) & -1 & 0 & 0\\
e_2(x_{234}) & -e_1(x_{34}) & 1 & 0\\
e_3(x_{234}) & -e_2(x_{34}) & e_1(x_{4}) & -1
\end{array}\right)$
}
\hspace{.5in}
\subfigure[The matrix $B^{-1}$]{
$\left(\begin{array}{cccc}
1 & 0 & 0 & 0 \\
\tilde{e}_1(x_{234}) & -1 & 0 & 0\\
\tilde{e}_2(x_{34}) & -\tilde{e}_1(x_{34}) & 1 & 0\\
\tilde{e}_3(x_{4}) & -\tilde{e}_2(x_{4}) & \tilde{e}_1(x_{4}) & -1
\end{array}\right)$.}
\end{center}
\end{figure}

\end{example}

\subsection{Relating elementary and complete truncated symmetric functions}\label{subsec:I_equals_J}

The following two results are the core of our proof that $I_h=J_h$.  Together, they prove a strong relationship between truncated elementary symmetric functions and truncated complete symmetric functions.  The first lemma is the base case of the induction in Proposition~\ref{prop:crucial_identity}.

\begin{lemma}\label{lem:base_case_induction}
For any $0<r\leq n$, the following identity holds:
$$e_r(1,\ldots,r) = (-1)^r \cdot \tilde{e}_r (r+1,\ldots,n) + \sum_{t=1}^r (-1)^{t+1} \cdot e_{r-t} (t+1,\ldots,r) \cdot \tilde{e}_t(t,\ldots,n).$$
\end{lemma}

\begin{proof}
We first prove that the claim holds for the case $r=n$ and then induct (downward) on the value $r$.  Theorem~\ref{thm:Maximal_case_equality_of_I_and_J} proved that
$$e_n(1,\ldots,n) = \sum_{t=1}^n (-1)^{t+1} \cdot e_{n-t} (t+1,t+2,\ldots,n) \cdot \tilde{e}_t(t,t+1,\ldots,n).$$
The claim holds when $r=n$ because $(-1)^n \cdot \tilde{e}_n (n+1,n+2,\ldots,n)=0$ by our conventions. 

Assume that the claim holds for some $r\leq n$.  We now prove the claim for the function $e_{r-1}(1,\ldots,r-1)$.  First note that if $S \subseteq \{1,2,\ldots,n\}$, the cardinality $|S|=d$, and $i_0 \notin S$, then 
\begin{align}
e_{d}(S) &=  \frac{x_{i_0}}{x_{i_0}} \cdot \prod_{i \in S} x_i= \frac{e_{d+1}(S \cup \{i_0\})}{x_{i_0}}. \label{equ:elementary_quotient}
\end{align}
When $S=\{1,2,\ldots,r-1\}$ we obtain $e_{r-1}(1,\ldots,r-1) = \displaystyle \frac{e_r(1,\ldots,r)}{x_r}$ which equals
\begin{align*}
\displaystyle \frac{(-1)^r \cdot \tilde{e}_r (r+1,\ldots,n) + \sum_{t=1}^r (-1)^{t+1} \cdot e_{r-t} (t+1,\ldots,r) \cdot \tilde{e}_t(t,\ldots,n)}{x_r},
\end{align*}
the latter equality by the inductive hypothesis.  Rearranging the sum shows that the function $e_{r-1}(1,\ldots,r-1)$ equals
\begin{align*}
\displaystyle \frac{(-1)^{r+1} \cdot {e}_0 (r+1,\ldots,r) \cdot \tilde{e}_r(r,\ldots,n) + (-1)^r \cdot \tilde{e}_r (r+1,\ldots,n)}{x_r}\\
\hspace{1in}\displaystyle + \; \sum_{t=1}^{r-1} (-1)^{t+1} \cdot \frac{e_{r-t} (t+1,\ldots,r)}{x_r} \cdot \tilde{e}_t(t,\ldots,n).
\end{align*}
We use Equation~\eqref{equ:elementary_quotient} and the convention ${e}_0 (r+1,\ldots,r)=1$ in Definition~\ref{def:truncated_sym_fcn}:
\begin{align*}
e_{r-1}(1,\ldots,r-1) = (-1)^{r-1} \cdot \frac{\tilde{e}_r(r,\ldots,n) - \tilde{e}_r (r+1,\ldots,n)}{x_r}\hspace{.5in}\\
+ \; \sum_{t=1}^{r-1} (-1)^{t+1} \cdot e_{(r-1)-t} (t+1,\ldots,r-1) \cdot \tilde{e}_t(t,\ldots,n).
\end{align*}
Finally we apply Lemma~\ref{lem:decomposing_truncated_complete_symfcn} to the first summand of the previous equation:
\begin{align*}
e_{r-1}(1,\ldots,r-1) &= (-1)^{r-1} \cdot \tilde{e}_{r-1} (r,\ldots,n) \\
&\hspace{.25in} + \sum_{t=1}^{r-1} (-1)^{t+1} \cdot e_{(r-1)-t} (t+1,\ldots,r-1) \cdot \tilde{e}_t(t,\ldots,n).
\end{align*}
By induction, the claim is proven.
\end{proof}

The next proposition proves a similar relation, except that an arbitrary function $e_d(1,\ldots,r)$ takes the place of the function $e_r(1,\ldots,r)$.

\begin{proposition}\label{prop:crucial_identity}
For any $0 < d \leq r \leq n$, the following identity holds:
$$e_d(1,\ldots,r) = (-1)^d \cdot \tilde{e}_d (r+1,\ldots,n) + \sum_{t=1}^d (-1)^{t+1} \cdot e_{d-t} (t+1,\ldots,r) \cdot \tilde{e}_t(t,\ldots,n).$$
\end{proposition}

\begin{proof}
We will prove by a double induction, inducting both on the difference $r-d$ and on the value $r$.  Lemma~\ref{lem:base_case_induction} proves the proposition when $d=r$, namely when $r-d=0$.  Assume that for some $N$ with $n \geq N \geq 0$, the proposition is true for all $d,r$ with difference $r-d \leq N$.  

Now fix $N+1$ and consider $r$ with $r-d = N+1$.  Theorem~\ref{thm:Maximal_case_equality_of_I_and_J} proves the claim when $r=n$, for arbitrary $d$.  We induct (downward) on $r$, with $r=n$ as our base case.  Assume the claim holds for fixed $r$ with $r-d = N+1$.  It suffices to show the claim holds for the function $e_{d-1}(1,\ldots,r-1)$.

By Lemma~\ref{lem:decomposing_truncated_elem_symfcn}, we can rewrite the function $e_{d-1}(1,\ldots,r-1)$ as the quotient
$$\frac{e_d(1,\ldots,r) - e_d(1,\ldots,r-1)}{x_r},$$
which by the inductive hypothesis equals
\begin{align*}
\displaystyle \frac{(-1)^d \cdot \tilde{e}_d (r+1,\ldots,n) + \sum_{t=1}^d (-1)^{t+1} \cdot e_{d-t} (t+1,\ldots,r) \cdot \tilde{e}_t(t,\ldots,n)}{x_r} \hspace{.1in}\\ 
- \displaystyle \; \frac{(-1)^d \cdot \tilde{e}_d (r,\ldots,n) + \sum_{t=1}^d (-1)^{t+1} \cdot e_{d-t} (t+1,\ldots,r-1) \cdot \tilde{e}_t(t,\ldots,n)}{x_r}.
\end{align*}
Rearranging the numerators, we get $e_{d-1}(1,\ldots,r-1)$ equals
\begin{align*}
(-1)^d \cdot \frac{ \tilde{e}_d (r+1,\ldots,n) -  \tilde{e}_d (r,\ldots,n)}{x_r} \hspace{2.2in}\\
+ \; \sum_{t=1}^d (-1)^{t+1} \cdot \frac{e_{d-t} (t+1,\ldots,r) - e_{d-t} (t+1,\ldots,r-1)}{x_r} \cdot \tilde{e}_t(t,\ldots,n).
\end{align*}
Lemma~\ref{lem:decomposing_truncated_elem_symfcn} extends naturally to the case
\[ e_{(d-t)-1} (t+1,\ldots,r-1) = \frac{e_{d-t} (t+1,\ldots,r) - e_{d-t} (t+1,\ldots,r-1)}{x_r},\]
and hence we obtain
\begin{align*}
e_{d-1}(1,\ldots,r-1)  = (-1)^{d-1} \cdot \frac{\tilde{e}_d (r,\ldots,n) - \tilde{e}_d (r+1,\ldots,n)}{x_r} \hspace{.75in}\\
+ \; \sum_{t=1}^d (-1)^{t+1} \cdot e_{(d-t)-1} (t+1,\ldots,r-1) \cdot \tilde{e}_t(t,\ldots,n).
\end{align*}
Again applying Lemma~\ref{lem:decomposing_truncated_elem_symfcn}, we see
\begin{align*}
e_{d-1}(1,\ldots,r-1) = (-1)^{d-1} \cdot \tilde{e}_{d-1} (r,\ldots,n)  \hspace{1.8in} \\
+ \sum_{t=1}^d (-1)^{t+1} \cdot e_{(d-1)-t} (t+1,\ldots,r-1) \cdot \tilde{e}_t(t,\ldots,n).
\end{align*}
By induction, the claim holds for all $r$ with $r-d=N+1$, and hence for all $d,r$ with $0 < d \leq r \leq n$.
\end{proof}

We use the previous proposition for the following crucial observation.

\begin{corollary}\label{cor:elem_and_complete_symfcn_identity}
Given a Hessenberg function $h=(h_1,\ldots,h_n)$
\begin{align*}
(i) 	&\;\; e_d(1,\ldots,r)\in I_h \;\; \mbox{if and only if}\;\; \tilde{e}_d (r+1,r+2,\ldots,n)\in I_h, \mbox{and}\\
(ii) 	&\;\; \tilde{e}_d (r+1,r+2,\ldots,n)\in J_h \;\; \mbox{if and only if}\;\; e_d(1,\ldots,r)\in J_h.
\end{align*}
\end{corollary}

\begin{proof}
Each function $\tilde{e}_t(t,t+1,\ldots,n)$ from Proposition~\ref{prop:crucial_identity} lies in $J_{(n,n,\ldots,n)}$ since the maximal degree tuple $(n,n-1,n-2,\ldots,1)$ corresponds to the Hessenberg function $h=(n,n,\ldots,n)$.  Hence $\tilde{e}_t(t,t+1,\ldots,n)$ is in $I_{(n,n,\ldots,n)}$ by Corollary~\ref{cor:I_equals_J}.  Theorem~\ref{thm:I_ideal_containment} says that $I_h \supseteq I_{(n,n,\ldots,n)}$ for all $h$, so $\tilde{e}_t(t,t+1,\ldots,n) \in I_h$ for all $h$.  Moreover, Theorem~\ref{thm:J_ideal_containment} proves that the function $\tilde{e}_t(t,t+1,\ldots,n)$ lies in $J_h$ for all $h$. Thus Proposition~\ref{prop:crucial_identity} says 
\[e_d(1,\ldots,r) + (-1)^{d+1} \cdot \tilde{e}_d (r+1,r+2,\ldots,n) = f\]
where $f \in I_h \cap J_h$. The claim follows.
\end{proof}

\subsection{Proving \texorpdfstring{$I_h=J_h$}{I=J}}\label{I=J}
We now complete the proof that $I_h=J_h$ for each Hessenberg function $h$.  Our proof proceeds in two steps.  First, we show that the antidiagonal generators of $I_h$ are contained in $J_h$, from which we conclude that $I_h \subseteq J_h$.  Second, we show that the generators of $J_h$ are contained in $I_h$.  The proofs are very similar, but reverse the roles of elementary and complete truncated symmetric functions.  

The next claim relies on the identity in Lemma~\ref{lem:complete_symm_func_lemma}.

\begin{corollary}[$I_h \subseteq J_h$]\label{cor:I_in_J}
Let $h=(h_1,\ldots,h_n)$ be a Hessenberg function.  Fix $1 \leq j \leq n$, and let $\alpha_j = h_j-j+1$.  For each $j$, the generating function $e_{\alpha_j} (1,2,\ldots,h_j) \in \IAD$ lies in $J_h$.  We conclude that $I_h \subseteq J_h$.
\end{corollary}

\begin{proof}
Let $\beta=(\beta_n,\beta_{n-1},\ldots,\beta_1)$ be the degree tuple for $h$.  Choose any $j$ and let $r=h_j$.  We first prove that $\alpha_j \geq \beta_{r+1}$.  Consider the parts of the Hessenberg diagram of $h$ in Figure~\ref{fig:Hess_diag_for_JT_Corol} (and defined in Definition~\ref{def:Hess_diagram}).  Entry $(r+1,j)$ is not shaded because $r=h_j$.  This means $\beta_{r+1} < L$ where $L = r-j+2$ as indicated in Figure~\ref{fig:Hess_diag_for_JT_Corol}.  But
\[L = r-j+2 = h_j-j +2= \alpha_j +1\]
since $\alpha_j = h_j-j+1$.  It follows that $\beta_{r+1} \leq \alpha_j$.
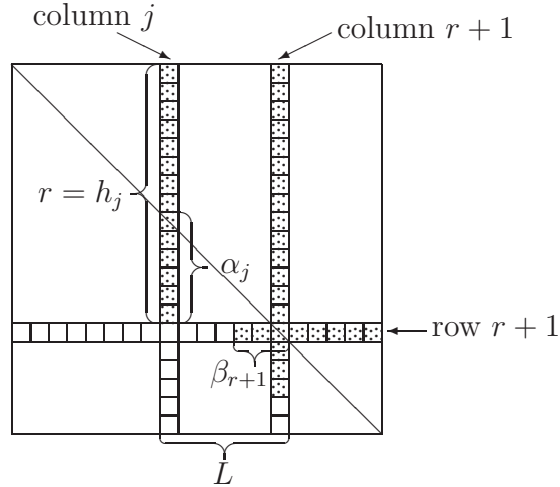
\begin{figure}[h]
\begin{center}
\begin{picture}(150,190)(0,-25)
\put(0,0){\line(1,0){140}}
\put(0,140){\line(1,0){140}}
\put(0,0){\line(0,1){140}}
\put(140,0){\line(0,1){140}}
\put(0,140){\line(1,-1){140}}

\put(0,35){\line(1,0){140}}
\put(0,42){\line(1,0){140}}
\multiput(0,35)(7,0){20}{\line(0,1){7}}
\put(157,39){\vector(-1,0){15}}
\put(159,37){row $r+1$}
\multiput(87,40)(4,0){3}{\circle*{1}}
\multiput(85,39)(4,0){4}{\circle*{1}}
\multiput(87,37)(4,0){3}{\circle*{1}}
\multiput(85,36)(4,0){4}{\circle*{1}}
\multiput(106,40)(4,0){9}{\circle*{1}}
\multiput(108,39)(4,0){9}{\circle*{1}}
\multiput(106,37)(4,0){9}{\circle*{1}}
\multiput(108,36)(4,0){9}{\circle*{1}}

\put(105,0){\line(0,1){140}}
\put(98,0){\line(0,1){140}}
\multiput(98,0)(0,7){20}{\line(1,0){7}}
\put(121,152){\vector(-2,-1){20}}
\put(123,150){column $r+1$}
\multiput(99,138)(0,-3){42}{\circle*{1}}
\multiput(101,139)(0,-3){42}{\circle*{1}}
\multiput(103,138)(0,-3){42}{\circle*{1}}

\multiput(63,0)(-7,0){2}{\line(0,1){140}}
\multiput(56,0)(0,7){20}{\line(1,0){7}}
\multiput(57,138)(0,-3){33}{\circle*{1}}
\multiput(59,139)(0,-3){33}{\circle*{1}}
\multiput(61,138)(0,-3){33}{\circle*{1}}
\put(40,152){\vector(2,-1){20}}
\put(8,155){column $j$}

\put(63,66){\oval(10,36)[tr]}
\put(63,60){\oval(10,36)[br]}
\put(74,67){\oval(12,8)[bl]}
\put(74,59){\oval(12,8)[tl]}
\put(79,61){$\alpha_j$}

\put(56,94){\oval(10,92)[tl]}
\put(56,88){\oval(10,92)[bl]}
\put(45,95){\oval(12,8)[br]}
\put(45,87){\oval(12,8)[tr]}
\put(10,88){$r=h_j$}

\put(89.5,35){\oval(11,7)[bl]}
\put(94.5,35){\oval(21,7)[br]}
\put(88,25.5){\oval(8,12)[tr]}
\put(96,25.5){\oval(8,12)[tl]}
\put(75,20){$\beta_{r+1}$}

\put(77,0){\oval(42,7)[bl]}
\put(84,0){\oval(42,7)[br]}
\put(76.5,-9.5){\oval(8,12)[tr]}
\put(84.5,-9.5){\oval(8,12)[tl]}
\put(76,-19){$L$}
\end{picture}
\end{center}
\caption{\label{fig:Hess_diag_for_JT_Corol}Hessenberg diagram schematic for Corollary~\ref{cor:I_in_J}.}
\end{figure}

We know that $\tilde{e}_{\beta_{r+1}}(r+1,r+2,\ldots,n) \in J_h$ by definition.  Theorem~\ref{thm:J_ideal_containment} showed that if $\alpha_j \geq \beta_{r+1}$ then $\tilde{e}_{\alpha_j}(r+1,r+2,\ldots,n) \in J_h$.  By definition $r=h_j$ so the function $\tilde{e}_{\alpha_j}(h_j+1,h_j+2,\ldots,n) \in J_h$.
The generator $e_{\alpha_j} (1,2,\ldots,h_j) \in \IAD$ also lies in $J_h$ by Corollary~\ref{cor:elem_and_complete_symfcn_identity}.(ii).  We conclude that $\IAD \subseteq J_h$ and hence  $I_h \subseteq J_h$ by Theorem~\ref{thm:min_gen_set_for_I}.
\end{proof}

We show that $J_h \subseteq I_h$ by a similar argument to the previous corollary.

\begin{corollary}[$J_h \subseteq I_h$]\label{cor:J_in_I}
Let $h=(h_1,\ldots,h_n)$ be a Hessenberg function.  For each $j$ with $1 \leq j \leq n$, the generator $\tilde{e}_{\beta_j}(j,j+1,\ldots,n)$ of $J_h$ lies in $I_h$.  We conclude that $J_h \subseteq I_h$.
\end{corollary}

\begin{proof}
Let $\beta=(\beta_n,\beta_{n-1},\ldots,\beta_1)$ be the degree tuple corresponding to $h$.  Fix $j$ and consider the generator $\tilde{e}_{\beta_j}(j,j+1,\ldots,n) \in J_h$.   By construction $\beta_j \leq j$.

If $\beta_j = j$ then $\tilde{e}_{\beta_j}(j,j+1,\ldots,n)=\tilde{e}_j(j,j+1,\ldots,n)$, which lies in $J_{(n,n,\ldots,n)}$.  Theorem~\ref{thm:Maximal_case_equality_of_I_and_J} proved that $J_{(n,\ldots,n)}=I_{(n,\ldots,n)}$. Theorem~\ref{thm:I_ideal_containment} proved $I_{(n,\ldots,n)} \subseteq I_h$ for all $h$. We conclude that if $\beta_j = j$ then $\tilde{e}_{\beta_j}(j,j+1,\ldots,n) \in I_h$.

Now assume that $\beta_j < j$.  By Corollary~\ref{cor:elem_and_complete_symfcn_identity}.(i), it suffices to show that  the function $e_{\beta_j}(1,\ldots,j-1)$ lies in $I_h$.  Recall that the set 
\[\mathfrak{C}_h = \{e_{h_i-t}(1,\ldots,h_i)\; | \; 0 \leq t \leq i-1\}_{i=1}^n\] 
generates $I_h$.  If $e_{\beta_j}(1,\ldots,j-1)$ is already in $\mathfrak{C}_h$ then the claim is trivially true.  If not, choose the largest $i$ such that $h_i \leq j-1$.    

We claim that $i=j-\beta_j$. Consider the parts of the Hessenberg diagram for $h$ in Figure~\ref{fig:Hess_diag_schematic} (and defined in Definition~\ref{def:Hess_diagram}).
\begin{figure}[h]
\begin{center}
\begin{picture}(150,190)(0,-25)
\put(0,0){\line(1,0){140}}
\put(0,140){\line(1,0){140}}
\put(0,0){\line(0,1){140}}
\put(140,0){\line(0,1){140}}
\put(0,140){\line(1,-1){140}}

\put(0,35){\line(1,0){140}}
\put(0,42){\line(1,0){140}}
\multiput(0,35)(7,0){20}{\line(0,1){7}}
\put(157,39){\vector(-1,0){15}}
\put(159,37){row $j$}
\multiput(63,40)(4,0){9}{\circle*{1}}
\multiput(65,39)(4,0){9}{\circle*{1}}
\multiput(63,37)(4,0){9}{\circle*{1}}
\multiput(65,36)(4,0){9}{\circle*{1}}
\multiput(106,40)(4,0){9}{\circle*{1}}
\multiput(108,39)(4,0){9}{\circle*{1}}
\multiput(106,37)(4,0){9}{\circle*{1}}
\multiput(108,36)(4,0){9}{\circle*{1}}

\put(105,0){\line(0,1){140}}
\put(98,0){\line(0,1){140}}
\multiput(98,0)(0,7){20}{\line(1,0){7}}
\put(121,152){\vector(-2,-1){20}}
\put(123,150){column $j$}
\multiput(99,138)(0,-3){42}{\circle*{1}}
\multiput(101,139)(0,-3){42}{\circle*{1}}
\multiput(103,138)(0,-3){42}{\circle*{1}}

\multiput(63,0)(-7,0){3}{\line(0,1){140}}
\multiput(49,0)(0,7){20}{\line(1,0){14}}
\multiput(51,139)(0,-3){28}{\circle*{1}}
\multiput(53,138)(0,-3){28}{\circle*{1}}
\multiput(55,139)(0,-3){28}{\circle*{1}}
\multiput(57,138)(0,-3){39}{\circle*{1}}
\multiput(59,139)(0,-3){40}{\circle*{1}}
\multiput(61,138)(0,-3){39}{\circle*{1}}
\put(33,152){\vector(2,-1){20}}
\put(-40,150){column $j-\beta_j$}
\put(59,157){\vector(0,-1){15}}
\put(38,160){column $j-\beta_j+1$}

\put(49,51){\oval(7,10)[tl]}
\put(49,48){\oval(7,10)[bl]}
\put(41.5,52){\oval(8,4)[br]}
\put(41.5,48){\oval(8,4)[tr]}
\put(33,48){$p$}

\put(49,116){\oval(7,48)[tl]}
\put(49,108){\oval(7,48)[bl]}
\put(39.5,116){\oval(12,8)[br]}
\put(39.5,108){\oval(12,8)[tr]}
\put(5,110){$j-\beta_j$}

\put(24,0){\oval(48,7)[bl]}
\put(32,0){\oval(48,7)[br]}
\put(24,-9.5){\oval(8,12)[tr]}
\put(32,-9.5){\oval(8,12)[tl]}
\put(12,-21){$j-\beta_j$}

\put(77,0){\oval(42,7)[bl]}
\put(84,0){\oval(42,7)[br]}
\put(76.5,-9.5){\oval(8,12)[tr]}
\put(84.5,-9.5){\oval(8,12)[tl]}
\put(76,-21){$\beta_j$}
\end{picture}
\end{center}
\caption{\label{fig:Hess_diag_schematic}Hessenberg diagram schematic for Corollary~\ref{cor:J_in_I}.}
\end{figure}
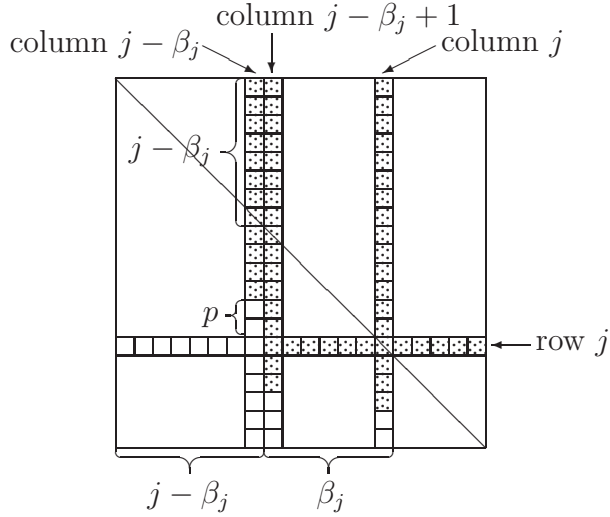
Since there are $\beta_j$ shaded boxes on and left of the diagonal in row $j$, we know both that $h_k \geq j$ for all $k \geq j-\beta_j+1$ and that column $j-\beta_j$ must have strictly less than $j$ shaded boxes.  Thus $i=j-\beta_j$.  

We now consider two cases.

\noindent\textbf{Case 1:} $h_i = j-1$

Since $\beta_j = j-i$,  the function $e_{\beta_j}(1,\ldots,j-1) = e_{j-i}(1,\ldots,j-1)$, which is one of the generators $\{e_{(j-1)-t}(1,\ldots,j-1)\}_{t=0}^{i-1} \subseteq \mathfrak{C}_h$ from the $i^{th}$ column of the $h$-Ferrers diagram.

\noindent\textbf{Case 2:} $h_i = (j-1)-p$ for some $p>0$

Since $h$ is a Hessenberg function we know $h_i \geq i$ and so $p \leq (j-1)-i$.  (This is also evident from Figure~\ref{fig:Hess_diag_schematic}.)  Thus we can apply Lemma~\ref{lem:cool_combinatorial_sym_func_lemma}:
\begin{align}e_{\beta_j}(1,\ldots,j-1) = \sum_{t=0}^p e_t(h_i+1,h_i+2,\ldots,j-1)\cdot e_{\beta_j-t}(1,\ldots,h_i).\label{equ:cool_comb_lemma_applied}
\end{align}

We wish to show that as $t$ varies, the degrees $\beta_j-t$ in Equation~\eqref{equ:cool_comb_lemma_applied} appear in the $i^{th}$ column of the $h$-Ferrers diagram for each nonvanishing term.  That is, we need to show that 
$$((j-1)-p)-(i-1) \leq \beta_j-t \leq (j-1)-p$$
for each nonvanishing term.  Since both $\beta_j=j-i$ and $-p \leq -t$ we can conclude that $j-i-p \leq \beta_j - t$.  To establish the second inequality, recall that $e_{\beta_j-t}(1,\ldots,h_i)$ vanishes if $\beta_j-t > h_i$.  By definition $h_i=j-1-p$ so it suffices to consider $t$ for which $\beta_j-t \leq j-1-p$, as desired.

Hence each $e_{\beta_j-t}(1,\ldots,h_i) \in I_h$ and so $e_{\beta_j}(1,\ldots,j-1) \in I_h$.  By Corollary~\ref{cor:elem_and_complete_symfcn_identity}, the original generator $\tilde{e}_{\beta_j}(j,j+1,\ldots,n) \in J_h$ lies in $I_h$ as desired.
\end{proof}

We summarize the previous two results in the following theorem.

\begin{theorem}\label{thm:I_h=J_h}
Let $h=(h_1,\ldots,h_n)$ be a Hessenberg function.  Then $I_h = J_h$.
\end{theorem}


\section{Open questions}\label{section:questions}

\numberwithin{theorem}{section}

We close with some open questions.

The ideals $I_h$ extend the Tanisaki ideal for the regular nilpotent Springer variety to regular nilpotent Hessenberg varieties.  Can we do this for arbitrary nilpotent Hessenberg varieties?

\begin{question}
Can we simultaneously generalize the Tanisaki ideal $\mathcal{I}_{\lambda}$ and the ideals $I_h$ to a two-parameter family $I_{\lambda,h}$, whose quotient $\mathbb{Q}[x_1,\ldots,x_n]/I_{\lambda,h}$ is the cohomology of the Hessenberg variety for $\lambda$ and $h$?
\end{question}

Together, this paper and~\cite{Mb1} show that there is a module isomorphism between the quotients $R/I_h$ and the cohomology of regular nilpotent Hessenberg varieties.  Does the quotient $R/I_h$ actually describe the ring structure of the cohomology, as it does for Springer varieties?

\begin{question}
Is there a ring isomorphism between $\mathbb{Q}[x_1,\ldots,x_n]/I_h$ and the cohomology of the regular nilpotent Hessenberg varieties (with rational coefficients)?
\end{question}

If the answer to the previous question were yes, it would mean that the cohomology of the regular nilpotent Hessenberg varieties is ring isomorphic to a family of smooth Schubert varieties, as described in the Introduction \cite{Din}, \cite{GasRei}.

The combinatorics of the truncated symmetric functions are interesting in their own right.  This raises several interesting questions, including:

\begin{question}
What are the matrices expressing the elementary truncated symmetric polynomials that generate $I_h$ in terms of the complete truncated symmetric polynomials that generate $J_h$, and vice versa (see Section~\ref{subsec:Grob_basis_for_elem_sym})?  Do they have properties similar to the change-of-basis matrices for traditional symmetric polynomials, as in Macdonald~\cite[Chapter I.6]{Mac95}?
\end{question}

\begin{question}
Biagioli, Faridi, and Rosas recently showed that their construction produces a minimal generating set for $\mathcal{I}_{\lambda}$ if the partition $\lambda$ is a hook~\cite{Biag07}.  In Section~\ref{sec:min_gen_set}, we construct a reduced generating set for $I_h$ that Galetto proved is minimal (Appendix~\ref{app:minimal_gen_proof}).  Do similar methods produce a minimal generating set for the Tanisaki ideal $\mathcal{I}_{\lambda}$ for arbitrary $\lambda$?
\end{question}


\newpage
\begin{appendix}
\setcounter{figure}{0}
\numberwithin{figure}{section}

\section{Example of generator-containment sequences for \texorpdfstring{$n=5$}{n=5}}\label{app:Two figures}

In Figure~\ref{fig:Hess_Fcn_Hasse_n=5}, we give the Hasse diagram on Hessenberg functions for $n=5$.  As in the $n=4$ diagram from Figure~\ref{fig:h_beta_trees}.(a), the double-lined dashed edge between two Hessenberg functions denotes a generator-containment sequence, described in detail in Section~\ref{subsec:special_containment_of_I_ideals}.


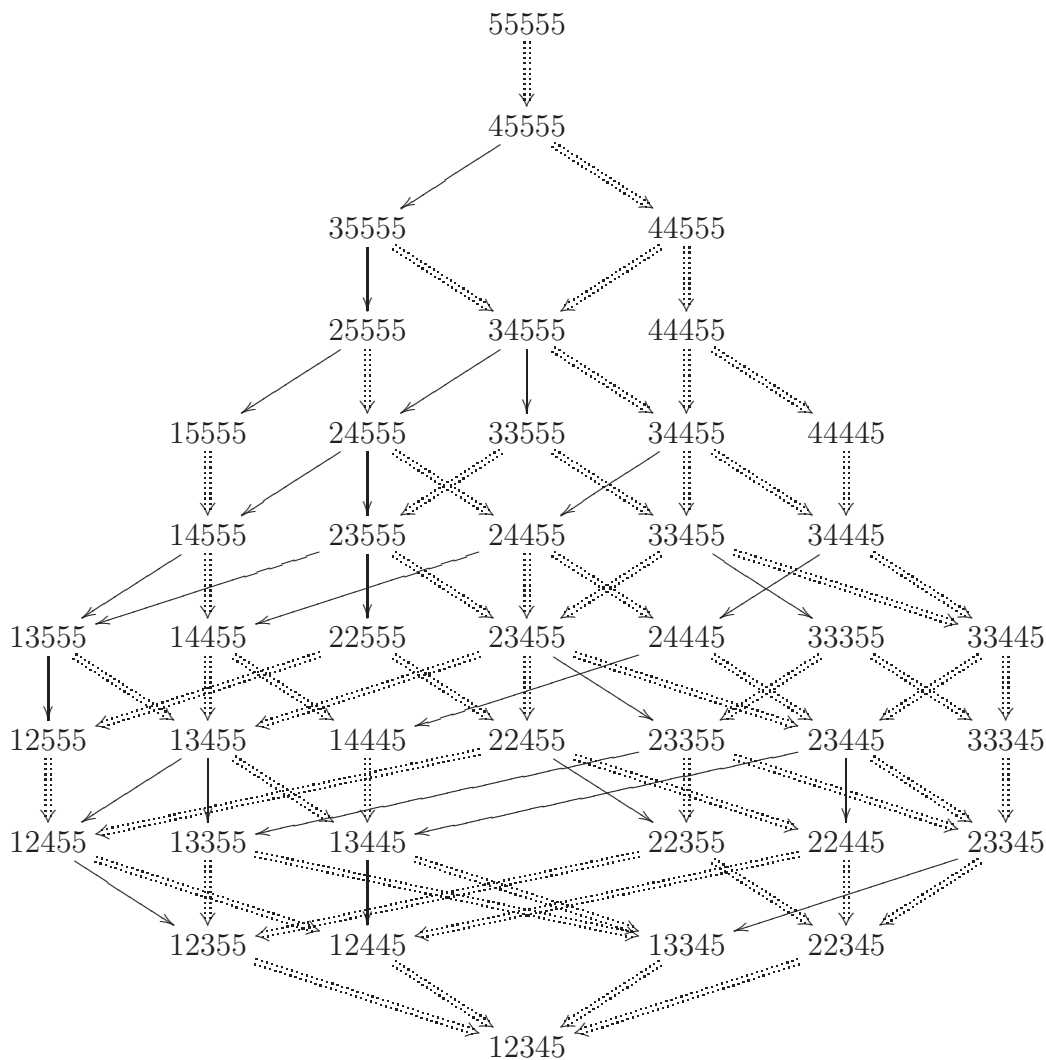
\begin{figure}[h]
$$
\xymatrix{
\\
\\
	&	&	& 55555\ar@{:>}[d]\\
	&	&	& 45555\ar[dl] \ar@{:>}[dr]\\
	&	& 35555\ar[d] \ar@{:>}[dr]	&	& 44555 \ar@{:>}[dl] \ar@{:>}[d]\\
	& 	& 25555\ar[dl] \ar@{:>}[d]	& 34555\ar[dl] \ar[d] \ar@{:>}[dr]	& 44455\ar@{:>}[d] \ar@{:>}[dr]\\
	& 15555\ar@{:>}[d] & 24555\ar[dl] \ar[d] \ar@{:>}[dr]	& 33555\ar@{:>}[dl] \ar@{:>}[dr] & 34455 \ar[dl] \ar@{:>}[d] \ar@{:>}[dr] & 44445\ar@{:>}[d]\\
	& 14555\ar[dl]\ar@{:>}[d] & 23555\ar[dll]\ar[d]\ar@{:>}[dr] & 24455\ar[dll]\ar@{:>}[d]\ar@{:>}[dr] & 33455\ar@{:>}[dl]\ar[dr]\ar@{:>}[drr] & 34445\ar[dl]\ar@{:>}[dr]\\
13555\ar[d]\ar@{:>}[dr]	& 14455\ar@{:>}[d]\ar@{:>}[dr] & 22555\ar@{:>}[dll]\ar@{:>}[dr] & 23455\ar@{:>}[dll]\ar@{:>}[d]\ar[dr]\ar@{:>}[drr] & 24445\ar[dll]\ar@{:>}[dr] & 33355\ar@{:>}[dl]\ar@{:>}[dr] & 33445\ar@{:>}[dl]\ar@{:>}[d]\\
12555\ar@{:>}[d] & 13455\ar[dl]\ar[d]\ar@{:>}[dr] & 14445\ar@{:>}[d] & 22455\ar@{:>}[dlll]\ar[dr]\ar@{:>}[drr] & 23355\ar[dlll]\ar@{:>}[d]\ar@{:>}[drr] & 23445\ar[dlll]\ar[d]\ar@{:>}[dr] & 33345\ar@{:>}[d]\\
12455\ar[dr]\ar@{:>}[drr] & 13355\ar@{:>}[d]\ar@{:>}[drrr] & 13445\ar[d]\ar@{:>}[drr] &	& 22355\ar@{:>}[dlll]\ar@{:>}[dr] & 22445\ar@{:>}[dlll]\ar@{:>}[d] & 23345\ar[dll]\ar@{:>}[dl]\\
	& 12355\ar@{:>}[drr] & 12445\ar@{:>}[dr] &	& 13345\ar@{:>}[dl] & 22345\ar@{:>}[dll]\\
	&	&	& 12345
}
$$
\caption{\label{fig:Hess_Fcn_Hasse_n=5} Hasse diagram on Hessenberg functions for $n=5$.}
\end{figure}

\newpage

\section{Galetto's proof of the minimality of the \texorpdfstring{$\IAD$}{I} generators}\label{app:minimal_gen_proof}

By considering the ideal $I^e$ generated by the generators of $\IAD$ over $\mathbb{Q} [x_1,\ldots,x_n]$ instead of over $\mathbb{Z} [x_1,\ldots,x_n]$, Galetto showed that every generating set of $I_h$ with $n$ generators is minimal.  The proof relies on some basic tools from commutative algebra, e.g. Matsumura's text~\cite{Matsu80}.


\newtheorem*{thm:Galetto_proof}{Theorem~\ref{thm:Galetto_proof}}
\begin{thm:Galetto_proof}[Galetto~\cite{Galetto}]
Every set of $n$ elements that generates $I_h$ is a minimal generating set for $I_h$.
\end{thm:Galetto_proof}

\begin{proof}
Let $I^e$ be the ideal of $R=\mathbb{Q} [x_1,\ldots,x_n]$ generated by the generators of $\IAD$. Since the polynomial ring $R$ is Cohen-Macaulay, we have:
	\[\height (I^e) = \dim (R) - \dim (R/I^e)\]
where $\height$ denotes the height of an ideal and $\dim$ denotes the Krull dimension of a ring \cite[Theorem 31]{Matsu80}. The equalities $$\IAD = I_h = J_h$$ imply that the ideal $I^e$ is generated by the generators of $J_h$. Corollary~\ref{thm:dim(R/J)} then implies that the $\mathbb{Q}$-vector space $R/I^e$ is finite dimensional; therefore $\dim (R/I^e) = 0$~\cite[pg.~14-15]{Matsu80}. Since $\dim (R) = n$ we conclude $\height (I^e) = n$. The minimal number of generators of an ideal is  at least the height of the ideal \cite[Theorem 18]{Matsu80}. The ideal $I^e$ is generated by $n$ elements and has height $n$, so a set of $n$ generators is minimal in $\mathbb{Q} [x_1,\ldots,x_n]$.

Now suppose the generators of $\IAD \subset \mathbb{Z} [x_1,\ldots,x_n]$ are not minimal. This means we may write one generator in terms of the others. The same relation holds over $\mathbb{Q} [x_1,\ldots,x_n]$, contradicting the minimality of the generators of $I^e$.
\end{proof}

\end{appendix}

\begin{acknowledgements}
The authors thank Richard Stanley for suggesting the terminology of truncated symmetric functions, Fred Goodman for helpful conversations and comments on early versions of this paper, and Frank Sottile and Megumi Harada for useful feedback.  We also thank Federico Galetto for a thorough reading of an earlier version of this paper on the arXiv and his proof in Appendix~\ref{app:minimal_gen_proof}.
\end{acknowledgements}


\end{document}